\documentclass[11pt,a4paper]{amsart}
\usepackage[foot]{amsaddr}

\usepackage{graphicx}
\usepackage[pdftex]{hyperref}
\usepackage{amssymb}
\usepackage{amsmath}
\usepackage{amsthm}
\usepackage{enumitem}
\usepackage{wrapfig}
\usepackage{color}
\usepackage[parfill]{parskip}
\usepackage[normalem]{ulem}

\numberwithin{equation}{section}
\newtheorem{theorem}{Theorem}[section]
\newtheorem{lemma}[theorem]{Lemma}
\newtheorem{proposition}[theorem]{Proposition}
\newtheorem{corollary}[theorem]{Corollary}

\theoremstyle{definition}
\newtheorem{definition}[theorem]{Definition}
\newtheorem{example}[theorem]{Example}
\newtheorem{remark}[theorem]{Remark}

\DeclareMathOperator{\Id}{Id}

\newcommand{\C}{\mathbb{C}}
\newcommand{\Z}{\mathbb{Z}}
\newcommand{\ol}[1]{\hat{#1}}
\newcommand{\iy}{\infty}

\newcommand\al{\alpha}
\newcommand\ga{\gamma}
\newcommand\tha{\theta}
\newcommand\wt{\widetilde}
\newcommand\thalf{\tfrac12}
\newcommand\app{\mathcal{A}}
\newcommand\cf{a}
\newcommand\fn{g}
\newcommand\bp{u}
\newcommand\cotr{\psi}
\newcommand{\defeq}{\mathrel{\mathop:}=}
\renewcommand{\paragraph}[1]{\par{#1}}

\begin{document}
\title{Approximation systems}

\author[V. A. Pessers]{Victor A. Pessers}
\email{vpessers@gmail.com}

\author[T. H. Koornwinder]{Tom H. Koornwinder}
\email{T.H.Koornwinder@uva.nl}
\address{Author 2:
Korteweg-de Vries Institute, University of Amsterdam,
P.O.~Box 94248,
1090~GE~Amsterdam, The Netherlands}


\begin{abstract}
We introduce the notion of an approximation system as a generalization of
Taylor approximation, and we give some first examples. Next we
develop the general theory, including error bounds and a
sufficient criterion for convergence. More examples follow. We conclude the article with a description of numerical implementation and directions for future research. Prerequisites are mostly elementary complex analysis.
\end{abstract}

\maketitle


\section{Introduction}
Holomorphic functions can be characterized by the property that they can
locally be approximated by a power series, viz.\ the Taylor
series. When we are given a holomorphic function $\fn$ on a simply connected open set $U\subset \C$,
and we denote its derivatives at a point $\bp\in U$ by
$\cf_i\defeq\fn^{(i)}(\bp)$, then $\fn$ is in fact the unique solution
for $\fn_0$ in the following infinite system of differential equations:
\begin{equation*}
\fn_i(\bp)=\cf_i,\qquad \fn'_i(x) = \fn_{i+1}(x)\quad(x\in U).
\end{equation*}
We could truncate
this infinite set of equations after $n$ steps, so that we get:
\begin{equation}\label{preMain}
\fn_i(\bp)=\cf_i\quad(i\leq n)
,\qquad \fn'_i(x) = \fn_{i+1}(x)
\quad(x\in U,\; i< n).
\end{equation}
In terms of integrals this can be restated as:
\begin{equation}\label{preMainIntegral}
\fn_i(x)= \cf_i + \int_{\bp}^x \fn_{i+1}(t)\,dt
\quad(x\in U,\;i< n).
\end{equation}
Here a holomorphic function $g_n$ such that $g_n(u)=a_n$ should be specified. In particular, if we pick $g_n$ to be constant $a_n$, the resulting solution for $g_0$ will be the $n$-th order Taylor polynomial.

In this paper we will introduce a generalization of the Taylor
approximation, called {\em approximation systems},
based on the following alteration of \eqref{preMain}:
\begin{equation}\label{alteration Taylor}
\fn_i(\bp)=\cf_i,\qquad \fn'_i(x) = f_i(x,\fn_{i+1}(x)),
\end{equation}
where each $f_i$ is a holomorphic function in two complex variables. To get the
generalized equivalent for the $n$-th order Taylor approximation, we truncate
again these equations after $n$ steps, so that rewriting them in terms of
integrals gives:
\begin{equation}\label{trunDE}
\fn_n(\bp)\defeq \cf_n,\qquad
\fn_i(x)\defeq \cf_i + \int_{\bp}^x f_i\big(t,\fn_{i+1}(t)\big)\,dt
\quad(x\in U,\;i<n).
\end{equation}
By picking the constant function $\cf_n$ for $\fn_n$, we get after $n$
integrations a solution for $\fn_0$. As was the case for Taylor
approximations, the solution to the truncated sequence of equations
\eqref{trunDE} serves as an approximation for the unique\footnote{The unicity will be proven in Proposition \ref{15}.} holomorphic function $\fn$
which satisfies the complete set of differential equations
\eqref{alteration Taylor}. As will be shown in Example \ref{TaylorApprox}, we
get the Taylor approximation by simply putting $f_i(x,y)=y$.

\paragraph{}

The concept of approximation system, introduced in this article, is thus a generalization of the Taylor
approximation, in the sense that it allows for a possibly non-linear
function $f_i$ at each single integration step. By its non-linear nature, the resulting approximations will in general not involve a series. 
So although our method can be regarded as a generalized Taylor approximation, it is not a generalized Taylor series (such as can be found in, e.g. \cite{HuSe,IsSt,OdSh}), nor does it include any of these generalized Taylor series as a special case.

The method can also be regarded as a generalization of Picard iteration, in the sense that if we
take $\cf_i=\cf$ and $f_i=f$, i.e.,
if we take them identical for each $i\geq 0$, then \eqref{preMain}
essentially reduces to a single differential equation. The $n$-th
order approximation to it then coincides with the $n$-th Picard
iteration, as will be shown in more detail in Example \ref{ODE}.

\paragraph{}

It should be noted that there are several earlier instances of
the term ``approximation system'' in literature which are unrelated to each
other and to its meaning here. See for instance \cite{KhBa} and~\cite{Si}.

The first ideas about approximation systems and many of the basic properties
and the examples were already formulated by the first author in his BSc Thesis
\cite{Pe1}. He introduces a related but different concept called expansion
system in his MSc Thesis \cite{Pe2}.

\subsection{Preliminary remarks}

\noindent Throughout this paper, we keep the assumption that all functions involved are holomorphic. Denote the function to be approximated by $\fn$ and its domain by $U$. It is always assumed that $U$ is an open, connected and simply connected set in the complex plane. We make these assumptions for convenience; for much of the theory they can be relaxed. For instance, if one considers the case in which $U$ is an open interval on the real line instead, this will essentially provide the same approximations as the ones defined on an open subset of the complex plane containing this real interval, like is the case for the Taylor series.

Also the requirement of holomorphy is not strictly necessary, but reasonable nonetheless. Later on we will see in more detail that the coefficients $\cf_i$ are completely determined by the behavior of $\fn$ at the base point $\bp$ (again, like is the case for the Taylor approximation). So if, for example, we would only require that the function $\fn$ is $C^\infty$, then there exists no longer a relationship between the provided approximations and $\fn$'s behaviour outside of a small neighbourhood of the base point $\bp$. This makes the algorithm particularly useful for analytic functions.

The assumption that the domain $U$ is open in $\C$ is not essential either, but it is made primarily for consistency in the conditions under which various theorems are
proven. If we would allow for holomorphic functions on a not necessarily open set $U$, this would each time have to be interpreted as that the function is holomorphic on some open neighbourhood of $U$.

\subsection{Contents}

\noindent In Section~\ref{SC:definition} we define approximation systems
as a framework for approximating holomorphic functions, and we will already encouter some particular examples. However, in order to prove the convergence of such approximation systems, we will have to develop some useful theorems first. This will be done in sections~\ref{SC:transformations} to \ref{SC:convergence criterion}. In Section~\ref{SC:transformations} we consider some ways in which an approximation system can be transformed to another approximation system, and how the resulting approximations are then related to each other. Next we prove in Section~\ref{SC:error estimate} a general error estimate for approximation systems, which can be regarded a generalization of the classical error estimate for the Taylor series.
Section~\ref{SC:convergence criterion} starts with a discussion of the
relationship between an approximation system and the derivatives of the approximated function at its base point. We use this relationship to formulate conditions that will guarantee that the provided approximations really converge to the intended function (Theorem \ref{TH:DomImpConv} and \ref{CorCritApp}). In Section~\ref{SC:examples} we will then consider some more examples of
approximation systems, for which we will then be able to prove specific estimates of the maximal error.
Section~\ref{numImpl} is a short section in which we describe how approximation systems can be numerically implemented. We conclude this paper with some directions for further research.

\subsection{Notation}
\begin{itemize}
\item $B(c,R)$ denotes the open disk in $\C$ with center $c$ and radius $R$.
\item Given a map $\phi:U\to U$, by $\phi^{\circ\,n}$ we denote the $n$-fold composition of this map, so $\phi^{\circ\,0}\defeq\Id$ and
$\phi^{\circ(n+1)}\defeq\phi\circ\phi^{\circ\,n}$ for $n\geq 0$.
\item $D_1f$ and $D_2f$ denote the derivative of a function $f$ of two variables with respect to its first and second argument, respectively.
\item $\|f\|_Y$ denotes the sup norm of a function $f$ restricted to the domain $Y$.
\end{itemize}

\section{Definition and first examples}\label{SC:definition}

\subsection{Definition of approximation systems}

First we give the technical definition of an approximation system
(AS for short).
This only specifies some data:
a sequence of functions with their domains,
a sequence of coefficients, and a special point.
Next we define for a given function $g$ on $U$ what it means for an
AS to be an approximation system for $g$.

\begin{definition}\label{DF:approximation system}
An {\em approximation system} $\app^r = \left(\{f_i\},\{\cf_i\},\bp\right)$ is a triple consisting of:\nopagebreak
\begin{itemize}
\item
A sequence $\{f_i\} = \{f_i\}_{i=0}^{r-1}$ of holomorphic functions $f_i\colon Y_i \to \C$ in two variables,
where $Y_i = U\times V_i\subset\C^2$ with $U$ and $V_i$ open
and $U$ simply connected, and where
$f_i(x,\,.\,)$ is non-constant for all $x\in U$.
\item A sequence of complex constants $\{\cf_i\} = \{\cf_i\}_{i=0}^r$, called the {\em coefficients}, such that $\cf_{i+1}\in V_i$ for $i<r$.
\item A special point $\bp\in U$, called the
\em base point of the approximation system.
\end{itemize}
\end{definition}

In the above definition, the open set $U$ and the number $r$ are properties of the approximation system that are implicitly determined by the sequence of functions $\{f_i\}_{i=0}^{r-1}$. The set $U$ is what we call the {\em domain of the approximation system}, and $r$ is called its {\em order}. This order can be any nonnegative integer\footnote{In case $r=0$, $\{f_i\}$ just becomes an empty sequence.}, but we will also allow the possibility of $r$ being infinite, in which case our
sequences must be interpreted as $\{\cf_i\}_{i=0}^\infty$ and
$\{f_i\}_{i=0}^\infty$, i.e. sequences with domain $\Z_{\geq 0}$.


\begin{definition}\label{1}
An approximation system $\app^r = \left(\{f_i\},\{\cf_i\},\bp\right)$ is said to be
an {\em approximation system for the function~$\fn$} if there
exists a sequence $\{\fn_i\}_{i=0}^r$ of holomorphic
functions on $U$
such that $g=g_0$, $\fn_{i+1}(U)\subset V_i$ for all $i<r$,
and the following equations hold\footnote{$U$ and $V_i$ are defined by the domain $Y_i$ of $f_i$.}:
\begin{align}
\fn_i(\bp) &= \cf_i &(i<r+1)\label{EQ coefficient}\\
\fn'_i(x) &= f_i(x,\fn_{i+1}(x)) &(x\in U,\;i<r)\label{MainEquation}
\end{align}
In that case, we more specifically call $\app^r$ an AS for the sequence $\{\fn_i\}_{i=0}^r$\,.
\end{definition}

\begin{remark}
If $\app^r = \left(\{f_i\},\{\cf_i\},\bp\right)$ is an AS, then it follows easily from Definition \ref{DF:approximation system} that for $n < r$ also the truncated $\app^n := \left(\{f_i\}_{i=0}^{n-1},\{\cf_i\}_{i=0}^n,\bp\right)$ is an AS. Moreover, if $\app^r$ is an AS for the sequence $\{\fn_i\}_{i=0}^r$, then $\app^n$ is an AS for the sequence $\{\fn_i\}_{i=0}^n$. We will often use this principle that an AS of certain order entails all its truncated AS's of arbitrary lower order. Thus, statements proven about $\app^r$ usually imply similar statements for all truncations $\app^n$ for $n<r$. Likewise, statements about an unbounded AS $\app^\infty$ usually take the form of statements about all its truncations $\app^n$ for $n\in \Z_{\geq 0}$.
\end{remark}

\begin{remark} \label{5}

Note that according to the definition of an AS for a sequence $\{\fn_i\}_{i=0}^r$, $\app^r$ describes a set of differential equations which is satisfied by the sequence of functions
$\{\fn_i\}_{i=0}^r$. Except from the relation between $\fn_i$ and $\fn_{i+1}$, there aren't any further relations imposed between
functions $\fn_i$ and $\fn_j$ for other $i$ and $j$, as would be the case for most kinds of differential equations. The whole set of
equations, as described in \eqref{MainEquation}, can also be
formulated as:
\begin{align}
\fn_r(u) &= a_r &(\mbox{in case}\ r<\infty)\notag\\
\fn_i(x) &= \cf_i + \int_{\bp}^x f_i\big(t,\fn_{i+1}(t)\big)\,dt &(x\in U,\;i<r)\label{6}
\end{align}

Now suppose that $\app^r$ is given as in Definition \ref{DF:approximation system}, without any of the functions $\fn_i$ (including $\fn=\fn_0$) as in Definition \ref{1}.
Let a holomorphic function
$\fn_n\colon U\to V_{n-1}$ be given with $\fn_n(\bp)=\cf_n$.
Then unique holomorphic functions $\fn_i$ ($i<n$) exist on a sufficiently small
simply connected open neighbourhood $U_0\subset U$ of $\bp$
such that the truncated $\app^n$ is an AS for for the sequence $\{\fn_i\}_{i=0}^n$\,.
The functions $\fn_i$ can recursively be constructed by \eqref{6}.
\end{remark}

\begin{remark}\label{10}
Suppose in Definition \ref{1} only the holomorphic functions $\{f_i\}$
and the base point of the AS are given, but not the coefficients $\{\cf_i\}$. If we then have a sequence $\{\fn_i\}_{i=0}^r$ satisfying \eqref{MainEquation}, then equation \eqref{EQ coefficient} can be used as a definition for the $\cf_i$.
With $\{\cf_i\}_{i=0}^r$ thus chosen, the sequence $\{\fn_i\}_{i=0}^r$
will satisfy the whole of Definition \ref{1}.
For this reason, when mentioning an AS with respect to a certain sequence $\{\fn_i\}$, we may omit explicating the values of the $\cf_i$, since they are implied by the $\fn_i$ after all. Furthermore, upon mentioning an approximation system $\app^r$, we will tacitly adopt the notation of Definition \ref{1} for denoting its components (namely $\{f_i\}$ and $\{\cf_i\}$) and for its base point (namely $\bp$), unless explicitly stated otherwise. Because in many cases the base point of an AS either can be understood
from context or has an abstract value $\bp$, it is usually not even necessary to mention the base point specifically. In case of possible confusion we will write $\fn\colon(U,\bp)\to \C$ to emphasize that $\bp$ is the base point.
\end{remark}

Let us now apply the idea of
Remark \ref{5} by choosing an approximation for $\fn_n$, which we will call $\fn_n^{[n]}$, that is constantly equal to $\cf_n$. So suppose we are given $\app^r$ as an AS for $\{\fn_i\}$. To approximate $\fn\defeq \fn_0$, we first put $\fn_n^{[n]}\defeq \cf_n$ as an approximation for $\fn_n$, where $n$ is a fixed number such that $0\leq n<r+1$.
For $i<n$, we may define $\fn_i^{[n]}$ recursively using equation \eqref{6}, i.e.:
\begin{align}
\fn_n^{[n]}(x) &\defeq\cf_n &(x\in U)\\
\fn_i^{[n]}(x) &\defeq\cf_i + \int_{\bp}^x f_i\big(t,\fn_{i+1}^{[n]}(t)\big)\,dt &(x\in U,\;i<n)\label{Approximation}
\end{align}
provided this makes sense because of the inclusions
\begin{equation}\label{2}
\fn_{i+1}^{[n]}(U)\subset V_i\qquad(i<n)
\end{equation}
being valid.
As the $n$-th order approximation of $\fn$, we finally put:
\begin{equation*}
\fn^{[n]}\defeq\fn_0^{[n]}.
\end{equation*}

It should be emphasized that, throughout the whole construction of $\fn^{[n]}$, we have not used any information about the sequence $\{\fn_i\}$, except from the information already contained in the components of $\app^r$. Hence these approximations are uniquely determined by $\app^r$ itself, and are thus independent of the sequence $\{\fn_i\}$. Nonetheless, in view of approximation theorems to be given later in this paper, when $\app^r$ is an AS for $\{\fn_i\}$, then the functions $\fn^{[n]}$ usually turn out to approximate the function $\fn$ as $n$ gets bigger. This motivates the following summarizing definition.

\begin{definition}\label{3}
$\app^r$ is called
a {\em proper approximation system
for order $n$} if condition \eqref{2} holds at each step in the construction of the functions $\fn_i^{[n]}$ defined by
\eqref{Approximation}. In that case, $\fn^{[n]}$ is called the $n$-th order approximation of $\app^r$.
If $\app^r$ is a proper approximation system for every order $n<r+1$, we simply call
it a {\em proper approximation system} (or shortly PAS).
\end{definition}

\begin{remark}\label{RM A^n AS for g^[n]}
Given an AS $\app^r$ which is proper for order $n<r+1$, then it follows directly from the construction of $\fn^{[n]}$ that the truncated AS $\app^n$ is a proper approximation system for $\fn^{[n]}$. In fact, $\fn^{[n]}$ can be characterized as the unique function $\wt g$ such that $\app^n$ is an AS for $\wt g$ and $\wt g_n$ is constant.
\end{remark}

\begin{remark}
In most cases, the domains can be chosen in such a way that condition
\eqref{2} is satisfied. In general, for given $n$, we can always
shrink $U$ to a smaller open simply connected neighborhood
of $\bp$ such that \eqref{2} is satisfied.
\end{remark}
\goodbreak
\subsection{Examples of approximation systems}

\begin{example}\label{InvFun}
Suppose we have a holomorphic function $\fn\colon U\to \C$. Then
every sequence of invertible holomorphic functions $\{f_i\}_{i=0}^{r-1}$
in one variable (i.e. $f_i(x,y)$ only depends on $y$) with
sufficiently large range can be used for an approximation system for
$\fn$. The sequence $\{\fn_i\}_{i=0}^r$ is uniquely determined by $\fn_0 =
\fn$ and $\fn_{i+1}\defeq f_i^{-1}\circ \fn'_i$.

For instance, if we take $\fn(x) = e^x$ and
$f_i(x,y)={y^{i+2}}/{(i+1)!}$, then one may verify that this leads to an AS where $\fn_i$ is determined to be $e^{x/(i+1)!}$. This (somewhat artificial) example, where $f_i$ is a power of varying degree, may serve to demonstrate that in general there is much freedom when it comes to picking the functions $f_i$. In practice however, we will focus on AS's where the function $f_i$ stays more or less the same, except for an accompanying coefficient which may vary with respect to $i$. 
\end{example}

\begin{example}\label{TaylorApprox}{\rm\bf Taylor approximation}\\
When in Example \ref{InvFun} we let $f_i(x,y)=y$ and $V_i = \C$ for all $i\geq 0$, this
gives rise to a PAS $\app^\infty$ for the function $\fn\colon U\to\C$, such
that $\fn_i = \fn^{(i)}$ (and therefore $\cf_i = \fn^{(i)}(\bp)$). Because
\[
\cf_i+\int_{\bp}^x \sum_{k=i+1}^n \cf_k\frac{t^{k-i-1}}{(k-i-1)!}\,dt =
\sum_{k=i}^{n} \cf_k\frac{(x-\bp)^{k-i}}{(k-i)!}
\]
and
\[
\fn_n^{[n]}(x)=\cf_n,
\]
we get by induction that
\[
\fn^{[n]}(x)=\fn_0^{[n]}(x)=\sum_{k=0}^{n} \cf_k\frac{(x-\bp)^k}{k!}
=\sum_{k=0}^{n} \fn^{(k)}(\bp)\,\frac{(x-\bp)^k}{k!}\,,
\]
which is the $n$-th order Taylor approximation.
\end{example}

\begin{remark}
The convergence of the Taylor series is commonly expressed by the formula:
\begin{equation}\label{EQ taylor abr}
\fn(x) = \fn(\bp)+\fn'(\bp) \frac{x-\bp}{1!}+\fn''(\bp) \frac{(x-\bp)^2}{2!} + \ldots
\end{equation}
where the lower dots stand for those terms that will come after the last given term. 


As we will investigate more closely later on, AS's also provide accurate approximations for holomorphic functions. We will often express the convergence of an AS in similar fashion as in \eqref{EQ taylor abr}.
The general formula is then given as a nesting of integrals and functions:
\begin{equation}\label{EQ nested form}
\fn(x) = a_0 + \int_{\bp}^x f_0 \Big(x_1,a_1 + \int_{\bp}^{x_1} f_1 \Big(x_2,a_2 + \int_{\bp}^{x_2} f_2 \Big(x_3,a_3 + \ldots\Big)dx_3\,\Big)dx_2\,\Big)dx_1
\end{equation}
For instance, taking the concrete example given at the end of Example \ref{InvFun}, we obtain the formula:
\begin{equation}\label{EQ exp powers}
e^x = 1+\int_0^x \Big(1+\int_0^{x_1} \frac{1}{2!}\Big(1+\int_0^{x_2} \frac{1}{3!}\Big(1+\ldots \Big)^4
dx_3 \Big)^3dx_2 \Big)^2dx_1.
\end{equation}
Obtaining an $n$-th order approximation from such a formula, works quite similar to how one would obtain it from \eqref{EQ taylor abr}: we simply continue the formula up to the $n$-th level, and replacing the dots by $0$ at this level will then give us a precise expression for our $n$-th order approximation. Hence, formula \eqref{EQ exp powers} is simply another way of stating that the approximation system given by: 
$$f_i(x,y)=\frac{y^{i+2}}{(i+1)!}, \qquad a_i=1, \qquad u=0$$ converges to $e^x$. We will use such nested formula expressions more often in the rest of this article.
\end{remark}

\begin{example}\label{ODE}{\rm\bf Picard iteration}\quad\\
Consider the following ordinary differential equation (ODE)
on a domain $U$:
\begin{equation}\label{4}
\fn'(x)=f(x,\fn(x)),
\end{equation}
where $f\colon U\times V\to \C$ is a holomorphic
function. Suppose we have a solution $\fn$ to this differential equation
on the domain $U$ such that $\fn(\bp)$ equals a given value $a\in V$ at
a base point $\bp\in U$
and such that
$\fn(U)\subset V$. By letting $f_i\defeq f$ and $\cf_i\defeq\cf$, we obtain an AS
for the function $\fn$ by just putting $\fn_i \defeq g$ in \eqref{MainEquation}.
Now, for fixed $n$ and for $U$ a sufficiently
small open simply connected neighbourhood of $\bp$, the {\em Picard
iteration scheme} (see for instance
\cite[Ch.~1, \S3]{CoLe}):
\begin{equation}\label{Picard}
\left. \begin{aligned}
\fn_n^{[n]}(x) &\defeq\cf,\\
\fn_i^{[n]}(x) &\defeq\cf + \int_{\bp}^x f(t,\fn_{i+1}^{[n]}(t))\,dt\quad(i<n)
\end{aligned}
\quad \right\}\qquad
(x\in U).
\end{equation}
makes sense with $\fn_{i+1}^{[n]}(U)\subset V$\quad($i<n$). We recognize
\eqref{Approximation} and \eqref{2} specified for our example. Thus
our choice for the $f_i$ and $\cf_i$ gives us a PAS of any order $n$ in
the base point $\bp$, but with $U$ possibly dependent on $n$.
In fact, the general theory of ordinary
differential equations tells us that, for $U$ a sufficiently small
neigbourhood of $\bp$, we get a PAS for any order and we have
uniform convergence of $\fn^{[n]}$ to $\fn$ (see also Remark \ref{25}).

Note that it follows from \eqref{Picard} by induction with respect to
$n-i$ that we can say the following about a special case of
Definition \ref{3}:
\begin{equation}
\mbox{If\quad $f_i=f$, $\cf_i=\cf$ for all $i$\quad then\quad
$\fn_i^{[n]}=\fn_{i-1}^{[n-1]}=\ldots=\fn_0^{[n-i]}=\fn^{[n-i]}$.}
\label{16}
\end{equation}
\end{example}

\begin{example}\label{FunDifEq}{\rm\bf Functional differential equations}\\
As a generalization of the ODE \eqref{4}, consider the
functional differential equation (FDE)
\begin{equation}\label{GeneralFDE}
\fn'(x)=f(x,\fn\circ \phi(x)),
\end{equation}
on a domain $U$, where $f\colon U\times V\to \C$ is a
holomorphic function
and $\phi\colon U\to U$ is a holomorphic endomorphism on $U$.
Although this type of differential equations is even more general than the ODE in the previous example, it likewise gives rise to an AS for $\fn$ in the sense of Definition \ref{1}. We will use such FDE's later on in this article to provide a couple of examples of AS's for which we could obtain explicit error bounds.

Now let us assume that $\fn$ is a solution of this equation
with a given value $\fn(\bp)=\cf$ at a base point $\bp\in U$, such that
$\fn(\phi(U))\subset V$.
Put $\fn_i = \fn\circ \phi^{\circ\,i}\colon U\to V$, where we recall that $\phi^{\circ\,i}$ stands for the $n$-fold composition of $\phi$.
Choose $V_i\subset V_0\defeq V$ such that
$g\big(\phi^{\circ (i+1)}(U)\big)\subset V_i$.
Then we see that the functions
\begin{equation*}
f_i(x,y)\defeq\big(\phi^{\circ\,i}\big)'(x)\,f(\phi^{\circ\,i}(x),y(x))\qquad
(i\geq 0,\;(x,y)\in U\times V_i)
\end{equation*}
yield an AS for the functions $\{g_i\}_{i\geq 0}$. Indeed,
\eqref{MainEquation} can be seen to hold as follows.
If $\fn'_i(x) = f_i(x,\fn_{i+1}(x))$ for certain~$i$, then
\begin{align*}
\fn'_{i+1}(x)&= \phi'(x)\,\fn'_i(\phi(x))\\
&= \phi'(x)\,f_i\big(\phi(x),\fn_{i+1}(\phi(x))\big)\\
&= \phi'(x)\,\big(\phi^{\circ\,i}\big)'\big(\phi(x)\big)\,
f\big(\phi^{\circ(i+1)}(x),\fn_{i+2}(x)\big)\\
&= \big(\phi^{\circ(i+1)}\big)'(x)\,
f\big(\phi^{\circ(i+1)}(x),\fn_{i+2}(x)\big)\\
&= f_{i+1}(x,\fn_{i+2}(x)),
\end{align*}
so that the result follows by induction.

Now suppose that moreover $\phi(\bp)=\bp$. Then $\cf_i=\fn(\bp)=\cf$ for all $i$.
Also suppose that the AS obtained above is proper. Then we will show
that
\begin{equation}
\fn_i^{[n]}\circ\phi=\fn_{i+1}^{[n+1]},\quad
\mbox{in particular}\quad
\fn_i^{[n]}=\fn^{[n-i]}\circ\phi^{\circ\,i}.
\label{31}
\end{equation}
For the proof, first observe that in this situation \eqref{Approximation}
is equivalent to
\begin{equation}
\fn_i^{[n]}(\bp)=\cf,\qquad \big(\fn_i^{[n]}\big)'(x)=
\big(\phi^{\circ\,i}\big)'(x)\,f\big(\phi^{\circ\,i}(x),\fn_{i+1}^{[n]}(x)\big).
\label{32}
\end{equation}
The statement is trivially satisfied for $i=n$, so let us assume that $n-i\geq 1$ and moreover that the statement holds for all $m$ and $j$ such that $m-j<n-i$. Then
\begin{align*}
\big(\fn_i^{[n]}\circ\phi\big)'(x)
&= \phi'(x)\,\big(\fn_i^{[n]}\big)'(\phi(x))\\
&= \phi'(x)\,
\big(\phi^{\circ\,i}\big)'(\phi(x))\,f\big(
\phi^{\circ\,(i+1)}(x),\fn_{i+1}^{[n]}(\phi(x))\big)\\
&= \big(\phi^{\circ\,(i+1)}\big)'(x)\,f\big(
\phi^{\circ\,(i+1)}(x),\fn_{i+2}^{[n+1]}(x)\big)\\
&= \big(\fn_{i+1}^{[n+1]}\big)'(x).
\end{align*}
Since also
$\fn_i^{[n]}\circ\phi$ and $\fn_{i+1}^{[n+1]}$ have the same value $\cf$ at $\bp$,
we conclude that \eqref{31}) holds. Furthermore, by \eqref{32} we see that
\begin{equation}
\big(\fn^{[n]}\big)'(x)=f\big(x,\fn^{[n-1]}(\phi(x))\big).
\label{33}
\end{equation}
To give a concrete example of the above setting, consider the FDE given by:
\begin{equation}
\frac{d}{dx}e^x = \left(e^{x/p}\right)^p
\end{equation}

where $p$ is a positive integer. This fits into Equation \ref{GeneralFDE}, where $\fn(x) = e^x$, $f(x) = x^p$ and $\phi(x) = x/p$.
Working out the details for this particular example, one may verify that we get the following 
approximation:

\begin{equation}\label{EQ:NFEexpX}
e^x = 1+\int_0^x \Big(1+\int_0^{x_1} \frac{1}{p}\Big(1+\int_0^{x_2} \frac{1}{p^2}\Big(1+\ldots \Big)^p
dx_3 \Big)^p dx_2 \Big)^p dx_1.
\end{equation}

In Example \ref{ExamExp} it will be proven that this AS indeed converges on the entire complex plane.

\end{example}

\begin{remark}
Note that \eqref{33} together with $\big(\fn^{[n]}\big)(\bp)=\cf$ is equivalent
with
\begin{equation*}
\big(\fn^{[n]}\big)(x)=\cf+\int_{\bp}^x f\big(t,\fn^{[n-1]}(\phi(t))\big)\,dt.
\label{40}
\end{equation*}
Define an operator $S$ acting on holomorphic functions $h$ on $U$ satisfying
$h(\bp)=\cf$ and $h(\phi(U))\subset V$ as follows:
\begin{equation}
(Sh)(x)\defeq\cf+\int_{\bp}^x f(t,h(\phi(t)))\,dt.
\label{34}
\end{equation}
Then
\begin{equation}
\fn^{[n]}=S\big(\fn^{[n-1]}\big)=\ldots= S^n(a).
\label{35}
\end{equation}

Special cases of proper approximation systems coming from FDE's
with $\bp$ being a fixpoint of $\phi$
are given in Examples \ref{ExamExp}, \ref{ExamSinh}, \ref{ExamCosh}
and Remarks \ref{RemSin}, \ref{RemCos}. There it is possible
to generate the approximations $\fn^{[n]}$ by \eqref{35} using the operator
$S$. 

It is interesting to compare our operator $S$ given by \eqref{34}
and its iteration with the operator $T$ occurring in Grimm
\cite[Proof of Theorem 1]{Gr}. The FDE there is more complicated
than \eqref{GeneralFDE} ($f$ also depending on $\fn'$ and $\phi$ also depending
on $\fn$), but our FDE can be obtained as a special case. For that
case Grimm's operator $T$ becomes
\[
(Th)(x)\defeq f\big(x,\cf+{\textstyle\int_{\bp}^{\phi(x)} h(t)\,dt}\big).
\]
Then iterates $T^n(0)$ approximate $\fn'$ rather than $\fn$ and, for
$h(\bp)=\cf$, $S$ and $T$ are connected by
\[
(Sh)'=Th',\qquad (S^nh)'=T^n h',
\]
or in integral form:
\[
(Sh)(x)=\cf+\int_{\bp}^x (Th')(t)\,dt,\qquad
(S^nh)(x)=\cf+\int_{\bp}^x (T^n h')(t)\,dt.
\]
\end{remark}

\begin{remark}\label{DDE}
Consider the AS obtained by the choices $U=V=\C$,
$f(x,y)\defeq y$, $\phi(x)\defeq x+\al^{-1}\log\al$ ($\al\in(0,1)$), $\bp\defeq0$,
$\fn(x)\defeq e^{\alpha x}$ in Example \ref{FunDifEq}. It follows that
$f_i(x,y)=y$, $\fn_i(x)=\al^i e^{\al x}$, $\cf_i=\al^i$.
The AS is proper. One easily verifies that
\[
\fn_i^{[n]}(x)=\al^i\sum_{k=0}^{n-i}\frac{\al^k x^k}{k!}\,,\quad
{\rm hence}\quad
\fn^{[n]}(x)=\sum_{k=0}^n\frac{\al^k x^k}{k!}\,.
\]
Thus the $n$-th order approximation coincides with the $n$-th order
Taylor approximation. We have here a special case of Example \ref{FunDifEq}
where $\phi$ has no fixpoint and still everything can be worked out
explicitly.
\end{remark}

\section{Transformations of approximation systems}\label{SC:transformations}

The following proposition describes how a general holomorphic endomorphism $\psi\colon \ol U\to U$ into the domain of a given approximation system, naturally leads to another approximation system on the domain $\ol U$.

\begin{proposition}\label{Tran}
Let $\app^r=(\{f_i\},\{\cf_i\},\bp)$ be an AS for the sequence $\fn_i\colon(U,\bp)\to \C\ (i<r+1)$. Then a holomorphic mapping $\cotr\colon(\ol U,\ol \bp)\to (U,\bp)$ (mapping $\ol \bp$ to $\bp$) gives rise to a new AS ${\ol \app}^r = (\{\ol f_i\},\{\cf_i\},\bp) $ for the sequence $\ol \fn_i\colon (\ol U,\ol \bp)\to \C\ (i<r+1)$, where $\ol f_i(x,y)\defeq \psi'(x) f_i(\psi(x),y)$ and $\ol \fn_i(x)\defeq\fn_i\circ\cotr(x)$.
\end{proposition}
\begin{proof}
Equation \eqref{MainEquation} is satisfied because for $i<r$ we have
\[
\ol \fn_i'(x)=\psi'(x) \fn_i'(\psi(x))
=\psi'(x) f_i(\psi(x),\fn_{i+1}(\psi(x)))
=\psi'(x) f_i(\psi(x),\ol \fn_{i+1}(x))
=\ol f_i(x,\ol \fn_{i+1}(x)).
\]
The domain of $f_i$ can be chosen as $\ol U\times V_i$, and we have $\ol \fn_i(\ol U)\subset \fn_i(U)\subset V_i$, so that the new approximation system indeed satisfies the requirements.
\end{proof}

\begin{remark}\label{TranTwo}
Let $\app^r$, $\fn_i$, $\cotr$ and $\ol f_i$
be as in Proposition \ref{Tran}.
Assume that ${\ol \app}^r=(\{\ol f_i\},\{\cf_i\},\bp)$ is an AS
for some sequence
$\ol \fn_i\colon (\ol U,\ol \bp)\to \C$ ($i<r+1$),
where the $\ol \fn_i$ are not
necessarily as in Proposition \ref{Tran}.
If $\ol \fn_n=\fn_n\circ\cotr$ for some $n<r+1$ then by
Remark \ref{5} and by the proof of Proposition \ref{Tran}
we still have $\fn_i\circ\cotr = \ol \fn_i$ for all $i\leq n$.
\end{remark}

This Remark leads us immediately to the following Lemma, which will be used
in the next section.

\begin{lemma}\label{TranLem}
Let $\app^r=(\{f_i\},\{\cf_i\},\bp)$ be a PAS and
${\ol \app}^r=(\{\ol f_i\},\{\cf_i\},\bp)$ be an AS for respectively the sequences
$\fn_i\colon (U,\bp)\to \C$ and
$\ol \fn_i\colon (\ol U,\ol \bp)\to \C\ (i<r+1)$,
such that these are related through a holomorphic transformation
$\cotr\colon (\ol U,\ol \bp)\to (U,\bp)$ (i.e.,
$\ol \fn_i(x)=\fn_i\circ\psi(x)$ and
$\ol f_i(x,y)\defeq f_i(\psi(x),y)\,\psi'(x)\,)$.
Then ${\ol \app}^r$ is also a PAS, and $\ol \fn_i^{[n]}=\fn_i^{[n]}\circ \cotr$
for all $n<r$ and $i\leq n$.
\end{lemma}

It is also possible to transform the AS by linear transformations of
the second argument of the $f_i$, as is formulated in the following
Proposition. Its proof boils down to just a straightforward
calculation like in the
proof of Proposition \ref{Tran}.

\begin{proposition}\label{LinTran}
Let $\app^r=(\{f_i\},\{\cf_i\},\bp)$ be an AS
for the sequence $\fn_i\colon (U,\bp)\to \C$ $(i<r+1)$.
Given two sequences of complex numbers $b_i$ and $c_i$ $(i\geq 0$ and $c_i\neq 0)$, then
$\ol f_i(x,y)\defeq c_i\,f_i(x,c_{i+1}^{-1}(y-b_{i+1}))$ and $\ol \cf_i \defeq b_i+c_i \cf_i$, defines another AS ${\ol \app}^r$ for the sequence
$\ol \fn_i\colon (U,\bp)\to \C$ $(i<r+1)$,
where $\ol \fn_i(x)\defeq c_i\,\fn_i(x)+b_i$.
If $\app^r$ is proper, then so is ${\ol \app}^r$, and we have:
$\ol \fn_i^{[n]}(x) = c_i\, \fn_i^{[n]}(x)+b_i$.
\end{proposition}

\begin{remark}
As a consequence of Proposition \ref{LinTran}, every approximation system can be altered in such a way that all the coefficients $\cf_i$ vanish, namely by letting $b_i = -\cf_i$ and $c_i=1$. The approximations $\ol \fn_i^{[n]}$ thus obtained, simply differ a term $\cf_i$ from the approximations in the original setting. So in principle, we could do away with the coefficients $\cf_i$ in the Definition of approximation systems (definition \ref{DF:approximation system}, if one is willing to alter the target functions $\fn_i$ so that it is valued $0$ in the base point $\bp$. Although this might simplify some matters in the theory, we believe this would also result in a more circuitous practical implementation of approximation systems, especially when the target functions initially do not equal $0$ at the base point. Nonetheless, it is noteworthy that the possibility of leaving out the coefficients $\cf_i$ is there.
\end{remark}

\section{Error estimates of the approximation}\label{SC:error estimate}

In this section, we establish error estimates for the apprimations that arise from an approximation system (Theorem \ref{TH:ErrEst}). As a result of this (Corollary \ref{CR:Starlike}), we also obtain a simpler formula for the case that the domain of the approximated function is starlike with respect to our base point. Note that unlike is the case for Taylor approximations, the domain of convergence of an approximation system is in general not circular or even starlike with respect to the base point, as will be discussed in more detail in
Section \ref{SC directions}.

In the following propositions, we denote by $\Gamma(v,w)$ the set of piecewise
$C^1$-curves in our domain $U$ connecting $v$ with $w$, and we define
a metric $d_U$ on $U$ by
\[
d_U(v,w)\defeq\inf_{\ga\in\Gamma(v,w)} \ell(\ga),
\]
where $\ell(\ga)$ denotes the length of the path $\gamma$. Furthermore, given a path $\gamma\in\Gamma(v,w)$ with interval $I=[a,b]$ as its domain, let $\ell_\gamma(s)\ (s\in[a,b])$ denote the length of the initial part of $\gamma$ on the domain $[a,s]$. We now
prove the following simple lemma, which we will use in establishing the error estimate (Theorem \ref{TH:ErrEst}).

\begin{lemma}\label{LM:path}
Let $\ga\in\Gamma(v,w)$ be a curve in $U$. Then, for integral $n\geq 0$,
\begin{equation*}
\int_\ga d_U(v,t)^n |dt| \leq \frac{\ell(\ga)^{n+1}}{n+1}\,.
\end{equation*}
\end{lemma}
\begin{proof}
We have:
$$\int_\ga d_U(v,t)^n |dt| = \int_I d_U(v,\gamma(s))^n |d\gamma| \leq \int_I \ell_\gamma(s)^n |d\gamma| = \int_0^{\ell(\ga)} \theta^n d\theta = \frac{\ell(\gamma)^{n+1}}{n+1},$$
where we observe that in the second equality, applying the substitution $\theta = \ell_\gamma(s)$, we have that $d\theta$ equals $|d\gamma|$.
\end{proof}

\begin{theorem}\label{TH:ErrEst}
\begin{enumerate}[label=\bfseries\Alph*,leftmargin=0cm,itemindent=0.5cm,labelwidth=\itemindent]
\item Let $\app^r$ be an AS for the sequence
$\{\fn_i\}_{i=0}^r$, where we assume that $f_i$'s domain $Y_i = U\times V_i$ is such that $V_i$ is convex.
Assume moreover that $\app^r$ is a PAS of order $n<r+1$. Then
\begin{equation}\label{ErrEstA}
|\fn(x)-\fn^{[n]}(x)|\leq \|D_2f_0\|_{Y_0}\cdots \|D_2f_{n-1}\|_{Y_{n-1}}\,
\|\fn_n-\fn_n(\bp)\|_U\ \frac{d_U(\bp,x)^n}{n!}\quad(x\in U).
\end{equation}
\item Assume that moreover $n<r$. Then
\begin{equation}\label{ErrEstB}
|\fn(x)-\fn^{[n]}(x)|\leq \|D_2f_0\|_{Y_0}\cdots \|D_2f_{n-1}\|_{Y_{n-1}}\,
\|f_{n}\|_{U\times \fn_{n+1}(U)}\,\frac{d_U(\bp,x)^{n+1}}{(n+1)!}\quad(x\in U).
\end{equation}
\end{enumerate}
\end{theorem}

\begin{proof}
\begin{enumerate}[label=\bfseries\Alph*,leftmargin=0cm,itemindent=0.5cm,labelwidth=\itemindent]
\item We will show by downward induction with respect to $i$, starting at
$i=n$, that for $i\le n$:
\begin{equation}
\big|\fn_i(x)-\fn_i^{[n]}(x)\big|
\leq \|D_2f_i\|_{Y_i}\cdots \|D_2f_{n-1}\|_{Y_{n-1}}\,
\|\fn_n-\fn_n(\bp)\|_U\ \frac{d_U(\bp,x)^{n-i}}{(n-i)!}\quad(x\in U).
\label{7}
\end{equation}
Then the case $i=0$ of \eqref{7} yields \eqref{ErrEstA}.
Clearly, \eqref{7} holds for $i=n$ because
$\fn_n^{[n]}(x)=\cf_n=\fn_n(\bp)$.
Now suppose that for some $i<n$ \eqref{7} holds with $i$ replaced by $i+1$.
Then for $i$ we have the following string of (in)equalities:
\begin{align*}
& |\fn_i(x)-\fn_i^{[n]}(x)|
=\left|\int_{\bp}^x \big(f_i(t,\fn_{i+1}(t))-f_i(t,\fn_{i+1}^{[n]}(t))\big)\ dt\right|\\
&\qquad
\leq \inf_{\ga\in\Gamma(\bp,x)} \int_{\ga}
|f_i(t,\fn_{i+1}(t))-f_i(t,\fn_{i+1}^{[n]}(t))|\,|dt|\\
&\qquad
\leq \inf_{\ga\in\Gamma(\bp,x)}
\int_{\ga} \|D_2f_i\|_{Y_i}\,|\fn_{i+1}(t)-\fn_{i+1}^{[n]}(t)|\,|dt|\\
&\qquad
\leq \inf_{\ga\in\Gamma(\bp,x)}
\int_{\ga} \|D_2f_i\|_{Y_i}\,\|D_2f_{i+1}\|_{Y_{i+1}}\cdots
\|D_2f_{n-1}\|_{Y_{n-1}}\,\|\fn_n-\fn_n(\bp)\|_U\,
\frac{d_U(\bp,t)^{n-i-1}}{(n-i-1)!}\,|dt|.
\end{align*}
The first equality is by substitution of \eqref{6} and
\eqref{Approximation}.
In the second inequality we used the convexity of $V_i$ and in the third
inequality the induction hypothesis. Furthermore, by Lemma~\ref{LM:path},\begin{eqnarray*}
\inf_{\ga\in\Gamma(\bp,x)}
\int_{\ga} \frac{d_U(\bp,t)^{n-i-1}}{(n-i-1)!}\,|dt| \leq
\inf_{\ga\in\Gamma(\bp,x)} \frac{\ell(\ga)^{n-i}}{(n-i)!} =
\frac{d_U(\bp,x)^{n-i}}{(n-i)!}.
\end{eqnarray*}
The last equality follows by continuity and the definition of
$d_U$. Combining this result with the earlier inequalities, we conclude that \eqref{7} holds for $i$ as well, and thus the statement follows by induction.
\item
Now we may show by downward induction with respect to $i$, starting at $i=n$, that for $i\le n$:
\begin{equation}
\big|\fn_i(x)-\fn_i^{[n]}(x)\big|
\leq \|D_2f_i\|_{Y_i}\cdots \|D_2f_{n-1}\|_{Y_{n-1}}\,
\|f_n\|_{U\times \fn_{n+1}(U)}\ \frac{d(\bp,x)^{n+1-i}}{(n+1-i)!}
\quad(x\in U).
\label{8}
\end{equation}
Then the case $i=0$ of \eqref{8} yields \eqref{ErrEstB}.
Clearly, \eqref{8} holds for $i=n$ because
\begin{equation}
\big|\fn_n(x)-\fn_n^{[n]}(x)\big|=|\fn_n(x)-\fn_n(\bp)|=
\left|\int_{\bp}^x f_n(t,\fn_{n+1}(t))\,dt\right|\le
\|f_n\|_{U\times \fn_{n+1}(U)}\,d(\bp,x).
\label{9}
\end{equation}
The proof of the induction step is analogous to what we did in part A.
\end{enumerate}
\end{proof}

\begin{remark}
Because the sets $U,V_0,\ldots,V_n$ are open, the sup norms at the
right-hand side of \eqref{ErrEstA} and \eqref{ErrEstB} may possibly be
infinite, in which case the theorem becomes a trivial statement.
We will often use Theorem \ref{TH:ErrEst} in cases where the sets
$U,V_0,\ldots,V_n$ are bounded and where the holomorphic functions under
consideration extend to holomorphic functions on open neighbourhoods of
the closures of these sets. Then the sup norms are finite.
\end{remark}

\begin{remark}
Substitution of the inequality \eqref{9} for $|\fn_n(x)-\fn_n(\bp)|$ in
\eqref{ErrEstA} would have given a weaker form of the inequality
\eqref{ErrEstB}, with the denominator $(n+1)!$ replaced by $n!\,$.
\end{remark}

If the domain $U$ is starlike with respect to $\bp$ then
$d_U(\bp,x)=|x-\bp|$ for all $x\in U$. Hence we have the following
corollary to Theorem \ref{TH:ErrEst}:

\begin{corollary}\label{CR:Starlike}
\begin{enumerate}[label=\bfseries\Alph*,leftmargin=0cm,itemindent=0.5cm,labelwidth=\itemindent]
\item\label{CR Starlike A}
Assume the same conditions as in Theorem \ref{TH:ErrEst}, and further
assume that $U$ is starlike with respect to the base point $\bp\in U$. Then
\begin{equation}\label{EQ:StarlikeErrEstA}
|\fn(x)-\fn^{[n]}(x)|\leq \|D_2f_0\|_{Y_0}\cdots \|D_2f_{n-1}\|_{Y_{n-1}}
\,\|\fn_n-\fn_n(\bp)\|_U\ \frac{|\,x-\bp|^n}{n!}\quad(x\in U).
\end{equation}
\item\label{CR Starlike B}
If moreover $n<r$ then
\begin{equation}\label{EQ:StarlikeErrEstB}
|\fn(x)-\fn^{[n]}(x)|\leq \|D_2f_0\|_{Y_0}\cdots \|D_2f_{n-1}\|_{Y_{n-1}}\,
\|f_{n}\|_{U\times \fn_{n+1}(U)}\,\frac{|\, x-\bp|^{n+1}}{(n+1)!}\quad(x\in U).
\end{equation}
\end{enumerate}
\end{corollary}

\begin{example}
Consider the PAS based on the Taylor series as described in
Example \ref{TaylorApprox}. There $f_i(x,y) = y$ and $V_i=\C$ for all $i$.
Furthermore, $\fn_n=\fn^{(n)}$ and $\fn^{[n]}$ is the $n$-th order Taylor
approximation of $\fn$.
Then \eqref{EQ:StarlikeErrEstA} and \eqref{EQ:StarlikeErrEstB}
respectively give:
\begin{align*}
|\fn(x)-\fn^{[n]}(x)|&\le\|\fn^{(n)}-\fn^{(n)}(\bp)\|_U\,\frac{|x-\bp|^n}{n!}\,,\\
|\fn(x)-\fn^{[n]}(x)|&\le\|\fn^{(n+1)}\|_U\ \frac{|\, x-\bp|^{n+1}}{(n+1)!}\,.
\end{align*}
These estimates of the remainder term
coincide with familiar estimates in Taylor's Theorem.
Thus it makes good sense to consider this $\fn^{[n]}$ as an $n$-th order
approximation of $\fn$ which
generalizes the $n$-th order Taylor approximation.
\end{example}

\begin{remark} \label{25}
For a PAS with $r=\iy$ it would be desirable to have that
$\fn^{[n]}(x)\to \fn(x)$ pointwise or uniform on some neighhbourhood
of $\bp$ as $n\to\iy$. In the generality of Theorem \ref{TH:ErrEst}
this cannot be concluded from \eqref{ErrEstA} or \eqref{ErrEstB}.
However, things improve if we suppose moreover that all functions
$f_i$ are the same function $f$ on $Y_i=Y=U\times V$.
Then \eqref{ErrEstB} yields:
\begin{equation}
|\fn(x)-\fn^{[n]}(x)|\leq
\|f\|_Y\,\frac{\big(\|D_2f\|_Y\big)^n\,d(\bp,x)^{n+1}}{(n+1)!}\qquad(x\in U),
\label{18}
\end{equation}
which implies uniform convergence on $Y$ if $f$ and $D_2f$ are bounded on
$Y$.
Of course, this is still under the assumption that we are dealing
with a PAS. In particular, it is required that $\fn_{i+1}^{[n]}(U)\subset V$
for $i<n$. If moreover $\cf_i=\cf$ for all $i$ then the requirement
simplifies by \eqref{16} to $\fn^{[m]}(U)\subset V$ for $m<n$.
Then, if the AS is such that $\fn_i(u)=\cf$ for all $\fn$ and if we take for $V$
an open disk around $\cf$ and for $U$ an open disk around $\bp$ such that
$f$ and $D_2f$ are bounded on $U\times V$, then
we can shrink $U$ to a sufficiently small open disk around $\bp$ such that
\eqref{18} yields for all $n$ that $\fn^{[n]}(x)\in V$ for $x\in U$.
Of course, this is a classical argument which was used for the
convergence proof of Picard iteration, see Example \ref{ODE}.
The error estimate \eqref{18} is also classical in the context of
Picard iteration, see the last formula in
\cite[Ch.~1, \S3]{CoLe}.
\end{remark}

\section{Convergence criteria}\label{SC:convergence criterion}

In this section, we will develop some criteria for approximation systems, that will guarantee its convergence on a certain domain. To do so, we will first have to explore in what way the derivatives of a function $g$ are determined by an approximation system for this function.

\subsection{Derivatives determined by the approximation system}

\begin{lemma}
\label{polnonnegcoef}
Let $\app^r$ be an AS for $\{\fn_i\}_{i=0}^r$. Then we have for $i<r$ and $n\ge 1$ that:
\begin{equation}
\fn_i^{(n)}(x)=\sum_{k+l<n}
(D_1^kD_2^{l}f_i)(x,\fn_{i+1}(x))\
P_{k,l}^n(\fn_{i+1}'(x),\ldots,\fn_{i+1}^{(n-1)}(x)),
\label{17}
\end{equation}
where $P_{k,l}^n$ is a polynomial with nonnegative coefficients, independent of the specific \allowbreak function~$g_i$.
\end{lemma}

\begin{proof}
This is by induction with respect to $n$,
where the case $n=1$ is:
\[
\fn_i^{(1)}(x)=f_i(x,\fn_{i+1}(x)),
\]
and we further observe that the derivative of a polynomial with non-negative coefficients is again a polynomial with non-negative coefficients.
\end{proof}

As an immediate corollary we obtain:

\begin{lemma} \label{14}
Let $\app^r$ be an AS for $\{g_i\}_{i=0}^r$. Then for $i<r$, the values
$\fn_i'(\bp),\fn_i''(\bp),\ldots,\fn_i^{(n+1)}(\bp)$ are
uniquely determined by the values
$\fn_{i+1}(\bp),\fn_{i+1}'(\bp),\ldots,\fn_{i+1}^{(n)}(\bp)$ through Equation \eqref{17}.
\end{lemma}

What we meant here with {\em uniquely determined}, is that if $\app^r$ is also an AS for another sequence $\{h_i\}_{i=0}^r$, and we have $h_{i+1}^{(j)}(u)=g_{i+1}^{(j)}(u)$ for $j=0,\ldots,n$, then it follows that $h_{i}^{(j)}(u)=g_{i}^{(j)}(u)$ for $j=1,\ldots,n+1$.

\begin{proposition} \label{15}
Let $\app^r$ be an AS for $\{g_i\}_{i=0}^r$. Then for $i+j<r+1$, the values $\fn_i^{(j)}(\bp)$ are uniquely determined in terms of $\app^r$. In case $r=\infty$, we have that the complete functions $g_i$, and hence $g$, are all uniquely determined by $\app^r$.
\end{proposition}
\begin{proof}
Let $i$ and $j$ be such that $i+j<r+1$. By Lemma \ref{14} we see
that $\fn_i(\bp),\fn_i'(\bp),\allowbreak\ldots,\fn_i^{(j)}(\bp)$ are uniquely
determined by $\cf_i,\fn_{i+1}(\bp),\fn_{i+1}'(\bp),\ldots,\fn_{i+1}^{(j-1)}(\bp)$. By further iterations, these are seen to be uniquely determined by $\cf_i,\cf_{i+1},\ldots,\cf_{i+j-1},\fn_{i+j}(\bp)$, hence ultimately by
$\cf_i,\ldots,\cf_{i+j}$. For $r=\infty$, the statement follows from the fact that holomorphic functions are determined by all derivatives in a point.
\end{proof}

\begin{remark}\label{RM dependency}
When we say that a certain value or function is uniquely determined by an AS $\app^r$, this means in general that there is both a dependency on the functions $f_i$ and the coefficients $a_i$. In situations in which the $f_i$ are assumed to be fixed, these statements thus reduce to the unique determinacy in terms of the coefficients $a_i$ alone.
\end{remark}

\begin{corollary}\label{11}
Let us have a PAS $\app^r$
as in Definition \ref{3}. Then, for $n<r+1$,
\begin{equation}
\fn^{[n](j)}(\bp)=\fn^{(j)}(\bp)\qquad(j\le n).
\label{24}
\end{equation}
\end{corollary}

\begin{proof}
By Remark \ref{RM A^n AS for g^[n]}, the truncated AS $\app^n$ is an AS both for
$\{\fn_i\}_{i=0}^n$ and for $\{\fn_i^{[n]}\}_{i=0}^n$.
Then the result follows from Proposition \ref{15}.
\end{proof}

\begin{remark}
\label{37}
If $\app^r$ is a PAS for the
sequence of functions $\fn_i$, such that $\bp=0$,
$\cf_i=0$, $f_i(x,y)$ is even in $x$ and in $y$, and $\fn_i$ is odd,
then we can do better and we conclude that for $n<\thalf r+1$
\begin{equation}
\fn^{[n](j)}(0)=\fn^{(j)}(0)\qquad(j\le 2n).
\label{36}
\end{equation}
Indeed, it then follows by induction from \eqref{Approximation}
that all $\fn_i^{[n]}$ are odd. Next, by the proof of Proposition \ref{15}
(where the dependence on the $\cf_i$ now can be omitted since these are 0)
we have for $i+j<r+1$ and $j$ odd that
$\fn_i'(0),\fn_i^{(3)}(0),\ldots,\fn_i^{(j)}(0)$ are uniquely determined
by $\fn_{i+1}'(0),\fn_{i+1}^{(3)}(0),\ldots,\fn_{i+1}^{(j-2)}(0)$, and hence
by iteration by $\fn_{i+\frac12(j-1)}'(0)$, and hence uniquely determined by
the $f_i$. A similar reasoning applies to the $\fn_i^{[n]}$ for $j\le 2n-2i-1$.
Thus \eqref{36} follows.

Similarly, if in the PAS
$\app^r$ for the sequence of
functions $\fn_i$ we have that $\bp=0$,
$f_i(x,y)$ is odd in $x$, and $\fn_i$ is even,
then we can again conclude that for $n<\thalf r+\thalf$
\eqref{36} holds (for $j\le 2n+1$).
Indeed, it then follows by induction from \eqref{Approximation}
that all $\fn_i^{[n]}$ are even. Next, by the proof of Proposition \ref{15}
we have for $i+j<r+1$ and $j$ even that
$\fn_i(0),\fn_i^{(2)}(0),\ldots,\fn_i^{(j)}(0)$ are uniquely determined
by $\cf_i,\fn_{i+1}(0),\fn_{i+1}^{(2)}(0),\ldots,\fn_{i+1}^{(j-2)}(0)$, and hence
by iteration by $\cf_i,\cf_{i+1},\ldots,\cf_{i+\frac12 j}$. A similar
reasoning applies to the $\fn_i^{[n]}$ for $j\le 2n-2i$.
Thus the claim follows.

\end{remark}

\subsection{Domination of approximation systems}

For power series, a well-known criterion for convergence is that if the coefficients of a power series are bounded in absolute value by a second converging series, then the first series will converge as well. In the following, we will generalize this idea of a so-called {\em dominating} power series to the case of approximations systems.

\begin{definition}
An AS $\app^r$
is called {\em non-negative} if
\[
\cf_i\ge0\quad(i<r+1),\qquad
(D_1^kD_2^lf_i)(\bp,\cf_{i+1})\ge0\quad(i<r,\quad k,l\ge 0).
\]
A non-negative AS $\wt \app^r$
is said to {\em dominate} an AS $\app^r$
if
\[
|\cf_i|\le\wt \cf_i\quad(i<r+1),\qquad
|(D_1^kD_2^lf_i)(\bp,\cf_{i+1})|\le (D_1^kD_2^l\wt f_i)(\bp,\wt \cf_{i+1})\quad
(i<r,\quad k,l\ge 0).
\]
\end{definition}

\begin{proposition}\label{19}
If $\app^r$ is a non-negative AS
for $\{\fn_i\}_{i=0}^r$ then
\begin{equation}
\fn_i^{(j)}(\bp)\ge 0\qquad(i+j<r+1).
\label{21}
\end{equation}
In particular if $r=\infty$, then all $g_i$ have non-negative derivatives.
\end{proposition}

\begin{proof}
Let $i+j<r+1$. From Lemma \ref{polnonnegcoef} (in particular
equation \eqref{17})
and from the assumption that $\app^r$ is non-negative, it follows
that $\fn_i'(\bp),\ldots,\fn_i^{(j)}(\bp)\ge0$ if
$\fn_{i+1}'(\bp),\ldots,\fn_{i+1}^{(j-1)}(\bp)\ge0$, which on its turn holds if $\fn_{i+2}'(\bp),\ldots,\fn_{i+2}^{(j-2)}(\bp)\ge0$, etc. As we finally have $g_{i+j-1}'(\bp)=f_{i+j-1}(\bp,a_{i+j})\ge 0$, it ultimately follows that \eqref{21} holds.
\end{proof}

\begin{lemma} \label{20}
Let $\wt\app^r$ be a non-negative AS for $\{\wt \fn_i\}_{i=0}^r$ which dominates an AS
$\app^r$ for $\{\fn_i\}_{i=0}^r$.
\begin{enumerate}[label=\bfseries\Alph*,leftmargin=0cm,itemindent=0.5cm,labelwidth=\itemindent]
\item
If for certain $n<r+1$ and $m\ge 0$, we have that the inequality:
\begin{equation}\label{EQ inequality dom g}
|\fn_{i}^{(j)}(\bp)|\le \wt \fn_{i}^{(j)}(\bp)
\end{equation}
holds for $i=n$ and all $j\le m$, then this inequality also holds for $i=n-1$ and all $j\le m+1$.
\item Equation \eqref{EQ inequality dom g} holds for all $i$ and $j$ such that $i+j<r+1$.\label{LM inequality dom g B}
\item If Equation \eqref{EQ inequality dom g} holds for $i=n$ and all $j\ge 0$, then it holds for all $i\le n$ and $j\ge 0$. \label{LM inequality dom g C}
\item If moreover $\app^r$ is a PAS for order $n$ and $\wt g_n$ has only non-negative derivatives, then:
\begin{equation}
|\fn^{[n](j)}(\bp)|\le \wt \fn^{(j)}(\bp)\qquad(j\ge0).
\label{23}
\end{equation}\label{LM inequality dom g D}
\end{enumerate}
\end{lemma}

\begin{proof}\ \\
\vspace{-0.5cm}

\begin{enumerate}[label=\bfseries\Alph*,leftmargin=0cm,itemindent=0.5cm,labelwidth=\itemindent]
\item\label{LM dominated AS 1}
It follows by \eqref{17} that for $j\le j_0+1$:
\begin{align*}
\big|\fn_{i_0-1}^{(j)}(\bp)\big|&\le
\sum_{k+l<j}
\big|(D_1^k D_2^{l}f_{i_0-1})(\bp,\cf_{i_0})\big|\,
P_{k,l}^j\big(|\fn_{i_0}'(\bp)|,\ldots,|\fn_{i_0}^{(j-1)}(\bp)|\big)\\
&\le
\sum_{k+l<j}
\big(D_1^kD_2^{l}\wt f_{i_0-1}\,\big)(\bp,\,\wt \cf_{i_0})\,
P_{k,l}^j\big(\wt {g\,}_{\!i_0}'(\bp),\ldots,\wt \fn_{i_0}^{\,(j-1)}(\bp)\big)\\
&=\wt{g\,}_{\!i_0-1}^{(j)}(\bp).
\end{align*}
\item
Let $r_0$ be any finite number such that $r_0<r+1$. Then putting $i_0=r_0$ and $j_0=0$, iterated application of part \ref{LM dominated AS 1} proves the statement for all pairs $i$ and $j$ such that $i+j\le r_0$. So if $r$ is finite itself, the statement is proven by letting $r_0=r$. For infinite $r$ the statement follows as we can then choose $r_0$ arbitrarily large.
\item
From part \ref{LM dominated AS 1} it directly follows that if \eqref{EQ inequality dom g} holds for $i=n$ and $j\ge 0$, then it also holds for $i=n-1$ and $j\ge 0$. Hence by iteration, the statement will follow for all $i\le n$.
\item
First we have that
$|\fn_n^{[n]}(\bp)|=|\cf_n|\le\wt \cf_n=\wt \fn_n(\bp)$.
Furthermore, for $j>0$ we have
$|\fn_n^{[n](j)}(\bp)|=0\le\wt \fn_n^{(j)}(\bp)$.
Now \eqref{23} follows by applying \ref{LM inequality dom g C} to the truncated AS $\app^n$ for the sequence $\{g_i^{[n]}\}_{i=0}^n$.\qedhere
\end{enumerate}
\end{proof}

\begin{remark}
Note that by Proposition \ref{19}, the condition in part \ref{LM inequality dom g D} that $\wt g_n$ has non-negative derivatives is automatically satisfied when $r=\infty$.
\end{remark}

\begin{theorem}\label{TH:DomImpConv}
Let $\wt\app^\infty$ be a non-negative
AS for $\{\wt \fn_i\}_{i=0}^\iy$ which dominates a PAS
$\app^\infty$ for $\{\fn_i\}_{i=0}^\iy$.
If $B(u,R)\subset U$ for some $R>0$, then $g^{[n]}$ converges uniformly to $g$ on $B(u,R)$.
\end{theorem}
\begin{proof}
For $|x-\bp|<R$ we have
\begin{multline*}
|\fn(x)-\fn^{[n]}(x)|
=\Big|\sum_{j=0}^\iy\frac{\fn^{(j)}(\bp)-\fn^{[n](j)}(\bp)}{j!}\,
(x-\bp)^j\Big|\\
=\Big|\sum_{j={n+1}}^\iy\frac{\fn^{(j)}(\bp)-\fn^{[n](j)}(\bp)}{j!}\,
(x-\bp)^j\Big|
\le\sum_{j={n+1}}^\iy\frac{2\,\wt \fn^{(j)}(\bp)}{j!}\,|x-u|^j,
\end{multline*}
where we have used Corollary \ref{11} and Lemma \ref{20} \ref{LM inequality dom g B} and \ref{LM inequality dom g D}.
Now use that the radius of convergence of the power series of $\wt \fn(x)$ around $\bp$ is at least $R$.
\end{proof}

\begin{theorem} \label{CorCritApp}
Let $\app^\infty$ be a non-negative AS
for $\{\fn_i\}_{i=0}^\iy$ such that for some $R>0$ we have $U=B(\bp,R)$ and
$B(\cf_{i},\wt R_i)\subset V_{i-1}$, where $\wt R_i = \lim_{r\to R}\;g_i(u+r)-a_i$.
Then
$\app^\infty$ is also a PAS for
$\{\fn_i\}_{i=0}^\iy$, and $\fn^{[n]}$ converges uniformly to $\fn$ on $U$.
\end{theorem}

\begin{proof}
We first prove that $\app^\infty$ is a PAS for
$\{\fn_i\}_{i=0}^\iy$. Fix $n\ge 0$. We will show by downward induction
with respect to $i$, starting at $i=n$, that
$\fn_i^{[n]}(x)\le \fn_i(x)$ if $0\le x-\bp < R$ and
$|\fn_i^{[n]}(x)-\cf_i|\le \fn_{i}(\bp+r)-\cf_{i}$ if $|x-\bp| < R$, where the last inequality is strict when $\wt R>0$. From this it then follows that $g_i^{[n]}(U)\subset V_{i-1}$ for $i=1,\ldots,n$. Note that if $\wt R_i = 0$, then we still have that $g_i^{[n]}(U)=\{a_i\}\subset V_{i-1}$.

For $i=n$ we have $\fn_n^{[n]}(x)=\cf_n$ and the assertions are clear.
Suppose the assertions hold for $i$ replaced by $i+1$. Then
\[
\fn_i^{[n]}(x) =\cf_i + \int_{\bp}^x f_i\big(t,\fn_{i+1}^{[n]}(t)\big)\,dt.
\]
Hence, for $0\le x-\bp < R$ we have
\[
\fn_i^{[n]}(x)\le \cf_i + \int_{\bp}^x f_i\big(t,\fn_{i+1}(t)\big)\,dt=\fn_i(x).
\]
Also, for $|x-u|< R$ we have
\[
|\fn_i^{[n]}(\bp+x)-\cf_i|\le \lim_{r\to R}\fn_i^{[n]}(\bp+r)-\cf_i\
\le \lim_{r\to R}\fn_i(\bp+r)-\cf_i = \wt R_i,
\]
where the inequality is strict if $\wt R_i\neq 0$, for in that case $g_i$ cannot be constant $\cf_i$, so that $g_i$ is a strictly increasing function on $[u,u+R)$.

By Theorem \ref{TH:DomImpConv} and the observation that $\app^\infty$ is a non-negative AS which of course dominates itself, it follows that $\fn^{[n]}$ converges uniformly to $\fn$ on $U$.
\end{proof}

\begin{remark}
\label{38}
It is possible to apply Theorem \ref{CorCritApp} to a specific
class of FDE's from Example~\ref{FunDifEq}. With the notation and
conventions
from that example, assume that $\cf\geq 0$, $\phi(\bp)=\bp$,
and that $\phi$ and $f$ have non-negative derivatives
(the function $f$ in both its arguments) in respectively $\bp$ and
$(\cf,\bp)$. Then, it follows by the composition and product formulas
for power series, that
$f_i(x,y)=\big(\phi^{\circ\,i}\big)'(x)\,f(\phi^{\circ\,i},y(x))$ also
has non-negative derivatives
in $(\cf,\bp)$. Hence the AS associated to our FDE is a non-negative AS.
By Proposition \ref{19} it then follows that
the solution $\fn$ of \eqref{GeneralFDE} has non-negative derivatives in $\bp$.
Now let $U=B(\bp,R)$ and $V_i = B(\cf,\wt R_i)$, where $\wt R_i = \lim_{r\to R}\;g_i(u+r)-a_i = \lim_{r\to R}\;\fn\circ\phi^{i}(\bp+r)-a$ and $R>0$. Further assume that $f$, $D_2 f$ and $\phi'$ are all bounded on their domain.

The requirements for
Theorem \ref{CorCritApp} are now
satisfied, and so the AS associated to our FDE is a PAS of all
orders, whose approximations converge to $\fn$.
Therefore, application of Corollary \ref{CR:Starlike}\ref{CR Starlike B}
yields
\begin{equation*}\label{ErrEstFDE}
|\fn(x)-\fn^{[n]}(x)|\leq
\|\phi'\|_U\,\|(\phi^{\circ 2})'\|_U\cdots \|(\phi^{\circ n})'\|_U\;
\|D_2 f\|_{Y_0}\cdots \|D_2 f\|_{Y_{n-1}}\;
\|f\|_{Y_n}\,\frac{R^{n+1}}{(n+1)!}\,.
\end{equation*}
Since $\phi$ is an endomorphism on $U$ we have
\[
(\phi^{\circ i})'(x)=
\phi'(\phi^{\circ i-1}(x))\phi'(\phi^{\circ i-2}(x))\cdots \phi'(x)
\leq \|\phi'\|_U^i \quad (x\in U).
\]
This implies the more simple estimates
\begin{align}\label{SimErrEstFDE}
&|\fn(x)-\fn^{[n]}(x)|
\leq \big\|\phi'\big\|_U^{\frac{1}{2}n(n+1)}\;
\|D_2 f\|_{Y_0}\cdots \|D_2 f\|_{Y_{n-1}}\;
\|f\|_{Y_n}\,\frac{R^{n+1}}{(n+1)!}\\
&\qquad\leq \big\|\phi'\big\|_{U}^{\frac{1}{2}n(n+1)}
\big(\|D_2 f\|_{Y_0}\big)^n\,
\|f\|_{Y_0}\;\frac{R^{n+1}}{(n+1)!}\,,
\label{39}
\end{align}
where we recall that $U=B(u,R)$ and $Y_0=B(u,R)\times B(a,\wt R_0)$, where $\wt R_0 = \lim_{r\to R}\; g(u+r)-a$. We will use \eqref{SimErrEstFDE} and the more rough estimate \eqref{39} a couple of times in the next section.
\end{remark}

\section{Further examples}\label{SC:examples}
In some of the examples below
we will use the {\em Chebyshev polynomials} $T_p$ of the
first kind and $U_p$ of the second kind
(see \cite[\S10.11, (2), (22), (23)]{Er}):
\begin{align*}
T_p(\cos\tha)&\defeq\cos(p\tha),\qquad\quad\;\;
T_p(x)=\thalf p \sum_{k=0}^{\lfloor p/2\rfloor}
\frac{(-1)^k(p-k-1)!}{k! (p-2k)!}\,(2x)^{p-2k},
\\
U_p(\cos\tha)&\defeq\frac{\sin((p+1)\tha)}{\sin\tha}\,,\quad
U_p(x)=\sum_{k=0}^{\lfloor p/2\rfloor}
\frac{(-1)^k(p-k)!}{k! (p-2k)!}\,(2x)^{p-2k}.
\end{align*}
We will also use the polynomial $T_p^+$, which is defined as:
\begin{equation*}
T_p^+(x)\defeq i^{-p}\,T_p(ix),\qquad
T_p^+(x)=\thalf p\sum_{k=0}^{\lfloor p/2\rfloor}
\frac{(p-k-1)!}{k! (p-2k)!}\,(2x)^{p-2k}.
\label{28}
\end{equation*}
We deduce that
\begin{equation*}
T_p^+(\sinh x) = T_p^+(i^{-1}\sin(ix))=i^{-p} T_p(\sin(ix))=
i^{-p} T_p(\cos(\thalf\pi-ix))=i^{-p}\cos(\thalf p\pi-ipx),
\end{equation*}
which equals $\cos(ipx)$ for $p$ even and
$i^{-1}\sin(ipx)$ for $p$ odd. Hence
\begin{equation*}
T_p^+(\sinh x)=\begin{cases}
\cosh(px)&\mbox{($p$ even),}\\
\sinh(px)&\mbox{($p$ odd).}
\end{cases}
\label{29}
\end{equation*}

We start with some examples which are special cases of an AS
determined by a FDE, as has been described in Example \ref{FunDifEq}.
The underlying FDE's with their solutions are
\begin{align}
\fn'(x) &= \big(\fn(x/p)\big)^p, & &\fn(0)=1,\qquad
\fn(x)=e^x,\nonumber\\
\fn'(x) &= T^+_p\big(\fn(x/p)\big), & &\fn(0)=0,\qquad
\fn(x)=\sinh x&\mbox{($p$ even),}\label{FDESinh}\\
\fn'(x) &= (-1)^{p/2}\,T_p\big(\fn(x/p)\big), & &\fn(0)=0,\qquad
\fn(x)=\sin x&\mbox{($p$ even),}\label{FDESin}\\
\fn'(x) &= \sinh(x/p)\,U_{p-1}\big(\fn(x/p)\big), & &\fn(0)=1,\qquad
\fn(x)=\cosh x,\nonumber\\
\fn'(x) &= -\sin(x/p)\,U_{p-1}\big(\fn(x/p)\big), & &\fn(0)=1,\qquad
\fn(x)=\cos x.\nonumber
\end{align}
Note that in \eqref{FDESinh} and \eqref{FDESin} with $p$ odd, the only
holomorphic solution $\fn$ of the FDE with $\fn(0)=0$ would be the function
identically zero.

\begin{example}\label{ExamExp}
In this example we will prove the convergence of the AS which was concisely expressed as \eqref{EQ:NFEexpX}. So, more explicitly, let $p\in \Z_{\geq 1}$\,, $R>0$, and consider the AS
obtained by the choices $f(x,y)\defeq y^p$, $\phi(x)\defeq x/p$, $\bp\defeq0$,
$\fn(x)\defeq e^x$, $U\defeq B(0,R)$,
$V_i\defeq B(1,\exp(R/p^{i+1})-1)\subset B(0,\exp(R/p^{i+1}))$
in Example \ref{FunDifEq}. Note that $\phi$ leaves the base point 0
invariant. Thus $\cf_i=\cf=\fn(0)=1$ for all $i$.
Furthermore,
\[
\fn_i(x) = \exp(x/p^i)\quad(x\in U),\qquad
f_i(x,y)=p^{-i}y^p\quad((x,y)\in U\times V_i).
\]
By induction we see that $\fn_{n-k}^{[n]}$ is a polynomial of degree
$(p^k-1)/(p-1)$. Hence $\fn^{[n]}$ is a polynomial of degree
$(p^n-1)/(p-1)$.

The AS satisfies the conditions of Remark \ref{38}.
Hence, it is proper and
by \eqref{SimErrEstFDE} we get for $|x|<R$ the estimate
\begin{align*}
|\fn(x)-\fn^{[n]}(x)| &\leq \big\|\phi'\big\|_U^{\frac{1}{2}n(n+1)}\;
\|D_2 f\|_{Y_0}\cdots \|D_2 f\|_{Y_{n-1}}\;
\|f\|_{Y_n}\,\frac{R^{n+1}}{(n+1)!}\\
&\leq p^{-\frac{1}{2}n(n+1)} \exp(R/p^{n}) \frac{R^{n+1}}{(n+1)!} \prod_{j=0}^{n-1}p \exp((p-1) R/p^{j+1})\\
&= \frac{e^R R^{n+1}}{p^{\frac{1}{2}n(n-1)}\,(n+1)!}\,,
\end{align*}
which converges to $0$ as $n$ goes to infinity.


For $p=2$, the first
four approximations are given by
\begin{align*}
\fn^{[0]}(x)&=1,\qquad
\fn^{[1]}(x)=1+x,\qquad
\fn^{[2]}(x)=1+x+\frac{x^2}{2}+\frac{x^3}{12}\,,\\
\fn^{[3]}(x)&=1+x+\frac{x^2}{2}+\frac{x^3}{6}+\frac{7 x^4}{192}
+\frac{x^5}{192}+\frac{x^6}{2304}+ \frac{x^7}{64512}\,.
\end{align*}
Note that formula \eqref{24} is indeed satisfied: the power series of
$\fn^{[n]}(x)$ and $\fn(x)=e^x$ coincide up to the term with $x^n$.
\end{example}

\begin{remark}

The approximations provided by the above AS converge fast as measured by the degree of $n$. For $p=2$ for instance, substituting $x=1$ in the 8th order approximation yields $\fn^{[8]}(1) = 2.71828182845902\ldots$ as approximation of $e$, which has 13 correct decimals. For the 8th order Taylor series on the other hand (corresponding to $p=1$), we get $\fn^{[8]}(1) = 2.71827\ldots$, which yields 4 correct decimals. This difference can be explained by the extra factor $p^{-\frac{1}{2}n(n-1)}$ in the error estimate, which for $p\geq 2$ will be more dominant than the factorial. Hence, for a given value of $x$ and $p\geq 2$, one may expect that the number of correct digits of $g^{[n]}(x)$ is roughly proportional to $n^2$ rather than $n\log n$. Of course, such a comparison does not take into account that there is also more calculation involved for higher values of $p$, especially if one would carry them out symbolically (see also Section \ref{numImpl}).
\end{remark}

\begin{remark}
From the above error estimate, one may also notice that we get a convergent approximation if we let the degree $n$ stay a fixed number greater or equal to $2$, and instead let $p$ approach infinity. In the particular case of $n=2$ for example, we obtain that:
$$e^x = \lim_{p\to \infty}1+\int_0^x \left(1+\int_0^{x_1}\frac{1}{p}\,dx_2\right)^p dx_1 = \lim_{p\to \infty}1+\int_0^x \left(1+\frac{x_1}{p}\right)^p dx_1.$$
Deriving the first and third expression of these three, noting that in this case the order of taking the limit or the derivate can be interchanged, we obtain the well-known formula:
$$e^x = \lim_{p\to \infty}\left(1+\frac{x}{p}\right)^p.$$
\end{remark}

\begin{example}\label{ExamSinh}
Let $p\in 2\Z_{\geq 1}$, and consider the AS
obtained by the choices $f(x,y)\defeq T_p^+(y)$, $\phi(x)\defeq x/p$, $\bp\defeq0$,
$\fn(x)\defeq\sinh x$, $U\defeq B(0,R)$,
$V_i\defeq B(0,\sinh(R/p^{i+1}))$
in Example \ref{FunDifEq}. Note that $\phi$ leaves the base point 0
invariant, so $\cf_i=0$ for all $i$.
Furthermore,
\[
\fn_i(x) = \sinh(x/p^i)\quad(x\in U),\qquad
f_i(x,y)=p^{-i}T_p^+(y)\quad((x,y)\in U\times V_i).
\]

The AS satisfies the conditions of Remark \ref{38}.
Hence, it is proper and
by \eqref{SimErrEstFDE} we get for $|x|<R$ the estimate

\begin{align*}
|\fn(x)-\fn^{[n]}(x)|&\leq
\big\|\phi'\big\|_U^{\frac{1}{2}n(n+1)}\;
\|D_2 f\|_{Y_0}\cdots \|D_2 f\|_{Y_{n-1}}\;
\|f\|_{Y_n}\,\frac{R^{n+1}}{(n+1)!}\\
&\leq p^{-\frac{1}{2}n(n+1)}
\frac{R^{n+1}}{(n+1)!}\;T_p^+(\,\sinh(R/p^{n+1}))\;
\prod_{j=0}^{n-1}{T_p^+}'(\,\sinh(R/p^{j+1}))\\
\noalign{\allowbreak}
&= p^{-\frac{1}{2}n(n+1)} \frac{R^{n+1}}{(n+1)!}\;\cosh(R/p^{n})\;
\prod_{ j=0 }^{n-1}\frac{p\,\sinh(R/p^j)}{\cosh(R/p^{j+1})}\\
\noalign{\allowbreak}
&= \frac{R^{n+1}\,\sinh R}{p^{\frac12 n(n-1)}\,(n+1)!}\;
\prod_{j=1}^{n-1}\tanh(R/p^j)\\
\noalign{\allowbreak}
&\leq \frac{R^{n+1}\,\sinh R}{p^{n(n-1)}\,(n+1)!}\,,
\end{align*}
where we use in the last inequality that $\tanh t\le t$ for $t\ge 0$. For $p=2$, we obtain the nested formula expression:
\begin{equation}
\sinh(x) = \int_0^x 1+2\Big(\int_0^{x_1} \frac{1}{2}+1\Big(\int_0^{x_2}\frac{1}{4}+\frac{1}{2}\Big(\ldots \Big)^2
dx_3 \Big)^2 dx_2 \Big)^2 dx_1.
\end{equation}
The first four approximations are in that case given by:
\begin{equation*}
\fn^{[0]}(x)=0,\qquad
\fn^{[1]}(x)=x,\qquad
\fn^{[2]}(x)=x+\frac{x^3}{6}\,,\qquad
\fn^{[3]}(x)=x+\frac{x^3}{6}+\frac{x^5}{120}+\frac{x^7}{8064}\,.
\end{equation*}
Here, even better than in Example \ref{ExamExp}, $\fn(x)$ and $\fn^{[n]}$
(which are both odd functions) have power series which agree up to the
(zero) term of degree $2n$. That this holds for all $n\geq 0$ was
explained in
Remark~\ref{37}.
\end{example}

\begin{remark}\label{RemSin}
By applying the transformation
$\psi\colon U \to U,\ {\psi(x) = ix}$ on the set $U = B(0,R)$
of the previous example,
we obtain by Proposition \ref{Tran}
and Lemma \ref{TranLem}
a new PAS with respect to the function sequence
$\hat \fn_k(x)=\fn_k\circ \cotr(x)=\sinh(ix/p^k)=i\sin(x/p^k)$,
where $\hat f_k(y,x)=\cotr'(x)\,f_k(y,\cotr(x))=ip^{-k}T^+_p(y)$.
Next by applying Proposition \ref{LinTran} with $a_i=-i$ and $b_i=0$,
we get a PAS with respect to $\check \fn_k(x)=\sin(x/p^k)$,
where $\check f_k(y,x)=p^{-k}\,T^+_p(iy)=(-1)^{p/2}\,p^{-k}\,T_p(y)$
($p$~is still assumed to be even).
The error estimate of the induced approximations can again be given
by Corollary \ref{CR:Starlike}.
It turns out to be equal to the estimate given in Example \ref{ExamSinh}.
\end{remark}

Next, we will see an example of an approximation system for which the functions $f_i(x,y)$ depend on both $x$ and $y$, rather than on $y$ alone.

\begin{example}\label{ExamCosh}
Consider the AS
obtained by the choices
$f(x,y)\defeq\sinh(x/p)\,U_{p-1}(y)\ (p\geq 2)$, $\phi(x)\defeq x/p$, $\bp\defeq0$,
$\fn(x)\defeq\cosh x$, $U\defeq B(0,R)$,
$V_j\defeq B(1,\cosh(R/p^{j+1})-1)$
in Example~\ref{FunDifEq}. As before, $\phi$ leaves the base point 0
invariant, so $\cf_j=1$ for all $j$.
The AS is a non-negative system (use \cite[10.11(27), 10.9(3)]{Er}) and it
satisfies the conditions of Remark \ref{38}.
Hence by \eqref{39} we have the estimate
\begin{equation}
|\fn(x)-\fn^{[n]}(x)| \leq p^{-\frac{1}{2}n(n+1)}\;
\big(\|D_2 f\|_Y\big)^n\;
\|f\|_Y\,\frac{R^{n+1}}{(n+1)!}\qquad(x\in B(0,R)),
\end{equation}
where $Y={B(0,R)\times B(1,\cosh(R/p-1))}$. Again by the positivity of the system we have
\begin{align*}
\|f\|_Y&=f(R,\cosh(R/p))\qquad =
\sinh(R/p)\,U_{p-1}(\cosh(R/p))=\sinh R,\\
\|D_2f\|_Y&=(D_2f)(\cosh(R/p),R)=
\sinh(R/p)\,U_{p-1}'(\cosh(R/p)).
\end{align*}
Also observe that
\[
\sinh t\;U_{p-1}'(\cosh t)=
\frac d{dt}\left(\frac{\sinh(pt)}{\sinh t}\right)
=\frac{\sinh(pt)}{\sinh t}\,(p\coth(pt)-\coth t)
<\frac{\sinh(pt)}{\sinh t}\,(p-1)\quad(t>0),
\]
since an elementary analysis yields that $p\coth(pt)-\coth t$ increases
from 0 to $p-1$ as $t$ runs from $0$ to $\iy$.
Hence
\[
\sinh(R/p)\,U_{p-1}'(\cosh(R/p))<
(p-1)\,\frac{\sinh R}{\sinh(R/p)}\,.
\]
Altogether,

\[
|\fn(x)-\fn^{[n]}(x)|<
\frac{R\sinh R}{p^{\frac12 n(n+1)}(n+1)!}\,
\left(\frac{(p-1)R\sinh R}{\sinh(R/p)}\right)^n
\qquad(x\in B(0,R)),
\]
which converges to $0$ as $n$ goes to infinity.
\end{example}

\begin{remark}\label{RemCos}
Somewhat similarly as in Remark \ref{RemSin}
we can apply a coordinate transformation
$\cotr\colon U \to U, x\mapsto ix$ to the set $U = B(0,R)$ of
the previous example.
Then we obtain by Proposition \ref{Tran}
and Lemma \ref{TranLem}
a new PAS with respect to the function sequence
$\hat \fn_j(x)=\fn_j\circ \cotr(x)=\cos(x/p^j)$,
where $\hat f_j(y,x)=\cotr'(x)\,f_j(y,\cotr(x))=-p^{-j}\sin(x/p^{j+1})
U_{p-1}(y)$.
Again, the error estimate of the induced approximations is the same
as the estimate given in Example \ref{ExamCosh}.
\end{remark}

\begin{example}\label{EX overconvergence}
Let $p\in\Z_{\geq 2}$, $R\in(0,1)$ and put  $\lambda_{p,i}\defeq(p^{i}+p-2)/(p^{i+1}-p^{i})$. Now consider an AS obtained by $f_i(x,y) \defeq \lambda_{p,i}\, y^p$, $u_i=0$, $\fn_i(x) \defeq (1-x)^{-\lambda_{p,i}}$, $U\defeq D(0,R)$ and $V_i\defeq B(1,(1-R)^{-\lambda_{p,i+1}}-1)$. Then we have a non-negative AS $\app^\infty$ with
$\cf_i = 1$. Application of Theorem \ref{CorCritApp} yields after some
computation that $\app^\infty$ is a PAS.
So by Corollary \ref{CR:Starlike}.{\bf B} we have for $x\in U$ that
\begin{align*}
|\fn(x)-\fn^{[n]}(x)| & \leq \|D_2 f_0\|_{Y_0}\cdots \|D_2 f_{n-1}\|_{Y_{n-1}}\;
\|f_{n}\|_{Y_{n}}\ \frac{R^{n+1}}{(n+1)!}\\
&\leq \lambda_{p,n} \left((1-R)^{-\lambda_{p,n}}\right)^{p}
\,\frac{R^{n+1}}{(n+1)!}\, \prod_{i=0}^{n-1}\ p\,\lambda_{p,i}
\left((1-R)^{-\lambda_{p,i}}\right)^{p-1}\\
&\leq \,
\frac{2^n R^{n+1}}{(1-R)^{2n}(n+1)!}\,,
\end{align*}
where we used the rough estimation that $\lambda_{p,i}\leq 1$ and $\lambda_{p,i}\leq 2/p$ for $i\ge 1$. This expression goes to $0$ as $n$ goes to infinity. As $R\in (0,1)$, This means that the approximation system converges at least on the open unit disk $B(0,1)$. Letting $p=3$, we have for instance for $|x|<1$:
\begin{equation*}
\frac{1}{1-x} =1+\int_0^x  \Bigl(1+\frac{2}{3}\int_{0}^{x_1} \Bigl(1+\frac{5}{9}\int_{0}^{x_2} \Bigl(1+\frac{14}{27}\int_{0}^{x_3} \Bigl(1+\ldots\Bigr)^3dx_4\Bigr)^3
dx_3 \Bigr)^{3}dx_2 \Bigr)^3 dx_1,
\end{equation*}
where we have moved the factors $\lambda_{p,i}$ in front of the integrals. In this case, the first four approximations are given by
\begin{align*}
\fn^{[0]}(x)&=1,\qquad
\fn^{[1]}(x)=1+x,\qquad
\fn^{[2]}(x)=1+x+x^2+\frac{4 x^3}{9}+\frac{2 x^4}{27},\\
\fn^{[3]}(x)&=1+x+x^2+x^3+\frac{127 x^4}{162}+\frac{751 x^5}{1458}+\frac{1850 x^6}{6561}+\frac{2500 x^7}{19683} + \ldots + \frac{1953125 x^{13}}{1087876733112}\,.
\end{align*}
\end{example}

\begin{remark}\label{RM domain of conv}
In all previous examples of this section, the domain of convergence equalled the whole complex plane, as the $R$ in the expression of the error estimation could be chosen arbitrarily large. This is no longer true for Example \ref{EX overconvergence}, as in its expression of the error estimation we must have $R<1$. This should not come as a surprise, as the original function has a singularity at $x=1$, which prohibits $U$ from containing a disk $D(0,R')$ with $R'>1$. However, since we are not dealing with Taylor approximations here, this does not mean that the actual domain of convergence must equal the open unit disk. What is interesting about Example \ref{EX overconvergence}, is that here the domain of convergence indeed seems to have a non-circular shape. The question what the true domain of convergence looks like, seems not so easy to answer, and it will be further discussed in Section \ref{SC directions}.
\end{remark}

\section{Numerical Implementation}\label{numImpl}

In all examples in the previous section, it was possible to calculate the
approximations provided by the algorithm explicitly. However, in many cases the
approximation $\fn_i^{[n]}$ is a polynomial whose degree is exponential
in the number $n-i$. Because the computation time of
each new step depends at least linearly on the input size of
$\fn_i^{[n]}$, also the computational effort tends to increase
exponentially as the order of the approximation increases.
For practical purposes, this can become problematic when $n$ gets large,
even when making use of computer software for symbolic computation.

It is therefore worth noting that the algorithm suits itself quite naturally for numerical implementation.
This can for instance be achieved as follows: suppose we want a numerical approximation
$\bar{\fn}^{[n]}$ of $\fn^{[n]}$ (which in its turn is an approximation of $\fn$)
along a path $\gamma$, going from $\bp$ to $x\in U$. The first step is to
partition $\gamma$ into a tuple
$\bar{\gamma}:\{0,1,\ldots,N\}\to \mathrm{Im}\ \gamma$,
consisting of only $N+1$ points from the original path $\gamma$.
Let\ $\Delta_k\defeq\bar\gamma(k+1)-\bar\gamma(k)$\ for $k=1,\ldots,N$.
Now we construct $\bar{\fn}_i^{[n]}:\{0,\ldots,N\}\to \C$ as follows:
\begin{align*}
\bar{\fn}_n^{[n]}(k) &= \cf_n &(k = 0,\ldots,N),\\
\bar{\fn}_i^{[n]}(0) &= \cf_i &(i = n-1,\ldots,0),\\
\bar{\fn}_i^{[n]}(k) &= \bar{\fn}_i^{[n]}(k-1)+f_i(\bar\gamma(k-1),\bar{\fn}_{i+1}^{[n]}(k-1)) \Delta_{k-1} &(k = 1,\ldots,N,\ \ i = n-1,\ldots,0).
\end{align*}
The vector $\bar{\fn}^{[n]}\defeq\bar{\fn}_0^{[n]}$
then represents a function with value
$\bar{\fn}^{[n]}(k)$ at the point $\bar{\gamma}(k)$.
This approximation of $\fn^{[n]}$ will become better
as the partition $\bar\gamma$ of $\gamma$ becomes finer.

Unlike exact computation of $\fn^{[n]}$,
computation of $\bar{\fn}^{[n]}$ carried out up to
a certain numerical precision does not get significantly more complicated
with each increment in $n$. Hence, provided that the coefficients $\cf_i$
are already given or computed, and that the functions $f_i$ to be
applied do not substantially increase in computational complexity, the
computational effort of this numerical implementation tends to be
linear in both the depth $n$ and the partition fineness $N$. However, one should keep in mind that unlike for symbolic evaluation, to get an increasingly accurate approximation of
a target function $\fn$ using the numerical algorithm, it requires
besides an increment in $n$ also a simultaneous (and not necessarily proportional) increment in $N$.

The need for numerical implementation also comes up
in situations where $\fn^{[n]}$ cannot even be expressed in terms of known
mathematical functions (such as we will encounter in the next section). Then, numerical calculation seems the only option.

\section{Directions for further research}\label{SC directions}

\subsection{The domain of convergence in Example \ref{EX overconvergence}}

As we noted in Remark \ref{RM domain of conv}, the fact that the domain of convergence in Example \ref{EX overconvergence} contains the unit disk, does not imply that it will be equal to it. Numerical simulations seem to indicate that indeed it is not, meaning that the actual domain of convergence (which we shall call $U_p^C$) can become considerably bigger than the open unit disk, although $x=1$ remains a boundary point of each $U_p^C$. This situation can be likened to the convergence behaviour in \cite{Ka}, where moreover the approximated function is the same. This phenomenon is known as {\em overconvergence}, and is an active field of research (cf. e.g. \cite{Ga,Wa}). The subscript $p$ in $U_p^C$ is necessary, as the set $U_p^C$ may differ for different values of the parameter $p$. Figure \ref{FI U_p} gives an impression of the extend of $U_p^C$, where $p$ ranges from $2$ to $6$. These estimations of the contours of $U_p^C$ have been obtained with the numerical algorithm from Section \ref{numImpl}, where the paths have been drawn radially outward from the base point $u$.

\begin{wrapfigure}{r}{0.5\textwidth}\label{FI U_p}
  \begin{center}
    \includegraphics[width=0.48\textwidth]{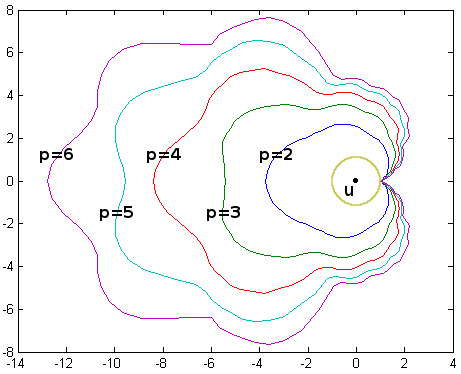}
  \end{center}
  \caption{Domain of convergence $U_p^C$.}
  \vspace{-10pt}
\end{wrapfigure}

We do not have certainty about how accurately our simulations are depicting the true domain of convergence, and a more advanced numerical analysis is certainly desirable. 
Of course, it would be best if we could fully determine $U_p^C$ through theoretical means, but it is questionable whether an explicit characterization of $U_p^C$ is easily obtainable. There are, nevertheless, also other questions by which we could improve our understanding of $U_p^C$. First, if we cannot exactly determine $U_p^C$, perhaps we can at least improve a bit on the result that $B(0,1)\subset U_p^C$. Is there some bigger region for which we can show that it is included in $U_p^C$? For which points we can show that they do not lie in $U_p^C$? Further, what can we say about the border of $U_p^C$? Is it smooth? Could it perhaps exhibit some kind of self-similarity? In this regard, it may be worth noting that if we would omit the integrations, i.e. if we would replace \eqref{Approximation} by $\fn_i^{[n]}(x) \defeq\cf_i + f_i\big(t,\fn_{i+1}^{[n]}(t)\big)$, then for suitable choices of $\cf_i$ and $f_i$, we have got a known recipe for making fractals.

\subsection{Fractional approximation systems}

So far, we have only considered approximation systems for which the functions $f_i$ and $g_i$ are all holomorphic on their domains, rather than meromorphic.
However, the assumption of holomorphy is just made for convenience.
In fact, the theory can also be developed under
the weaker assumption that $f_i$ and $g_i$ are meromorphic on their domains,
but with $f_i$ being regular in $(\bp,\cf_{i+1})$ and $g_i$ being regular in $\bp$.
Below we will consider so-called fractional approximation systems, which fit in such a wider theory. In the following discussion,
we will freely refer to earlier theorems in the paper, but under the suggestion that these remain valid in this more general setting.

\paragraph{}

Let $\fn$ be a meromorphic function on a region $U$ and holomorphic in the
base point $\bp:=0\in U$, and define sequences $\fn_i$ and $f_i:\C\times \C^*\to \C$ according to the following recursive construction:
$$f_i(x,y) \defeq \frac{x^{m_i}}{y},\qquad \fn_{i+1}\defeq \frac{x^{m_i}}{\fn_i'},\qquad
g_0:=g,$$
where $m_i$ denotes the order of the possible zero of $g_i'$ in the base point $\bp$ (in case $g_i'(u)\neq 0$, we simply have $m_i=0$). Note that even if $\fn$ would be holomorphic on $U$, there is the possibility that some of the $\fn_i$ become meromorphic, as a zero of $\fn_i'$ outside the base point $\bp$, will turn into a pole of $\fn_{i+1}$.
Such zeros might in principle show up arbitrarily close to the base point $\bp$ as $i$ goes to infinity, so restricting $U$ would not resolve this issue. By putting $\cf_i \defeq \fn_i(\bp)$, we see that every function $\fn$ holomorphic in $\bp$ gives rise to a unique AS $\app^\infty$. In this setting which is less restrictive than in Definition \ref{1}, we have that $\app^\infty$ is an AS for $\fn$. We will denote this kind of approximation systems by {\em fractional approximation systems}.

Taking $\fn$ respectively equal to $e^x$,\ $x^a$ and $\tan x$, this gives us in nested formula expression:
\begin{equation}
e^x\quad =\quad
1+\displaystyle\int_0^x\cfrac{dx_1}{
1-\displaystyle\int_0^{x_1}\cfrac{dx_2}{
1+\displaystyle\int_0^{x_2}\cfrac{dx_3}{
1-\displaystyle\int_0^{x_3}\cfrac{\quad\; dx_4\quad\;}{1+\ldots}}}}\label{EQ exp x frac as}
\end{equation}
\begin{equation}
x^a \quad =\quad
1+\displaystyle\int_1^x\cfrac{a\; dx_1}{
1+\displaystyle\int_1^{x_1}\cfrac{(1-a)\; dx_2}{
1+\displaystyle\int_1^{x_2}\cfrac{a\; dx_3}{
1+\displaystyle\int_1^{x_3}\cfrac{\ (1-a)\; dx_4\ }{1+\ldots}
}}}\label{EQ x^a frac as}
\end{equation}
\begin{equation}
\tan\ x\quad =\quad
\qquad \displaystyle\int_0^{x}\cfrac{dx_1}{
1-\displaystyle\int_0^{x_1}\cfrac{2\,x_2\,dx_2}{
1+\displaystyle\int_0^{x_2}\cfrac{4\,x_3\,dx_3}{
3-\displaystyle\int_0^{x_3}\cfrac{\ 28\,x_4\,dx_4\ }{5+\ldots}
}}}\label{EQ tan x frac as}
\end{equation}

For these continued fraction-like expression, we were able to do some simplifications by multiplying the numerators and denominators with appropriate factors (as can be done with normal continued fractions), which is in fact an application of Proposition \ref{LinTran}. In all of the above examples we only see rational coefficients $a_i$. This is to be expected for any function whose Taylor coefficients are rational, as follows implicitly from Lemma \ref{polnonnegcoef}. Note that even for relatively small $n$ (say $n\geq 4$) it already becomes impossible to express $g^{[n]}$ in terms of well-known mathematical functions. From a practical point of view, this means that we are more than ever dependent on numerical methods for evaluation of $g^{[n]}$. From a theoretical point of view however, this does not prohibit us from studying the theoretical convergence behaviour of these AS's.

\paragraph{}

The question now arises whether these fractional AS's also provide accurate approximations of the target function.
In \eqref{EQ exp x frac as} and \eqref{EQ x^a frac as}, the AS repeats itself after two steps (at least after the application of Proposition \ref{LinTran}, mentioned above). When it comes to the question of convergence, this repetitive behaviour could allow for a more conventional approach in which $g$ is viewed as a fixpoint of a certain operator (possibly one involving two or more subsequent integrations). This cannot apply, however, to the more general situation in which the AS is not repetitive, like is the case in \eqref{EQ tan x frac as} where the coefficients do not even follow a clear pattern.

\paragraph{}

Also, because we no longer require the functions involved to be holomorphic on their domain, singularities may easily show up in the (intermediate) approximations $g_{i}^{[n]}$. So, in order to make sense of these fractional approximation systems, let us make the following definition. We call a piecewise differentiable path $\gamma\colon [0,1]\to U$ with $\gamma(0)=u$, a convergent path for our AS, if for all sufficiently large $n$ we can evaluate $g^{[n]}$ along $\gamma$, i.e. the $g_i^{[n]}$ can all be recursively determined as long as we carry out all steps of our approximation system along $\gamma$. A domain of convergence could then be defined as a subset $U^C$ in $U$, such that for any point $p\in U^C$, we have that for some and for any convergent path $\gamma$ in $U^C$ which connects the base point $u$ with $p$, it holds that $g^{[n]}(p)$ (evaluated along $\gamma$) converges to $g(p)$.

\paragraph{}

In numerical simulations we have observed that these fractional AS's all seem to have a domain of convergence in the above sense. Example \eqref{EQ x^a frac as} exhibits an even more striking convergence behaviour, which clearly distinguishes the kind of convergence of fractional AS's from the convergence of Taylor series. By picking a non-integer value for $a$, the function $x^a$ will get a branched natural domain. The approximations provided by the fractional AS now seem to extend nicely to those other branches as well. This means that if we carry out the approximations along a path $\gamma$ circling once or more around the branch point, the function values of $g^{[n]}$ along $\gamma$ converge to the same values one would obtain from analytic continuation of $x^a$ along the path $\gamma$. This suggests that it would be possible (or even preferable) to define the notion of approximation systems on general Riemann surfaces, rather than merely on open subsets of the complex plane. In fact, as the domain of $x^a$ has just one branch point, we may in this particular case use the simple coordinate transformation $\psi(x) = e^x$ in Lemma \ref{TranLem}, to restate \eqref{EQ x^a frac as} as follows:
\begin{equation}
e^{a x} \quad =\quad
1+\displaystyle\int_0^x\cfrac{a\; e^{x_1}\; dx_1}{
1+\displaystyle\int_0^{x_1}\cfrac{(1-a)\; e^{x_1}\; dx_2}{
1+\displaystyle\int_0^{x_2}\cfrac{a\; e^{x_1}\; dx_3}{
1+\displaystyle\int_0^{x_3}\cfrac{\ (1-a)\; e^{x_1}\; dx_4\ }{1+\ldots}
}}}\label{EQ e^ax frac as}
\end{equation}
the domain of $e^{a x}$ being just $\C$. Again, it is suggested by numerical simulations that similar convergence behaviour as in \eqref{EQ x^a frac as}, is also present in other examples, including fractional AS's systems that are not repetitive.

Regarding the question in which cases and under what conditions these fractional AS's provide accurate approximations, further research is required in order to provide formal proofs and counterexamples.


\bibliography{1205.6370v3}
\bibliographystyle{habbrv}

\end{document}